\documentclass[a4paper, 10pt]{amsart}

\usepackage{amsmath,amssymb,hyperref,color}

\newtheorem{theorem}{Theorem}[section]
\newtheorem{lemma}[theorem]{Lemma}

\newtheorem{proposition}[theorem]{Proposition}

\theoremstyle{definition}
\newtheorem{claim}{}[theorem]

\DeclareMathOperator{\ord}{ord}
\DeclareMathOperator{\supp}{supp}

\numberwithin{equation}{section}

\title[The Characterization Problem for homocyclic class groups]{On the Characterization Problem for Krull Monoids with Homocyclic Class Group}

\author{Qinghai Zhong}

\address{Institute of Mathematics and Interdisciplinary Sciences, Xidian University, Xi'an 710126, China}
\email{zhong@xidian.edu.cn}
\urladdr{https://addcombin.github.io/zhong/}

\date{}

\begin{document}
\begin{abstract}
This paper addresses  a central problem in the factorization theory of Krull domains in which every class contains a prime divisor. It is well-known that the system of length sets depends on the class group only and the question is  whether this system carries enough information to recover the underlying class group. We give an affirmative answer for all homocyclic class groups $C_n^r$. More precisely, let $D$ be a Krull domain with class group $C_n^r$, with $n, r \in \mathbb N$,  and suppose that each class contains a prime divisor. If the system of length sets $\mathcal L (C_n^r)$ is equal to a system of length sets $\mathcal L (G)$, for any finite abelian group $G$ with $\mathsf D (G) \ge 4$, then $G \cong C_n^r$. The central tools are cross numbers of (minimal) zero-sum sequences over the class group.
\end{abstract}

\maketitle
\section{Introduction and main results}\label{s1}
Factorization theory has its origins in algebraic number theory, where
the ideal class group measures the failure of unique factorization in
rings of integers. A foundational result is Carlitz's theorem of 1960:
the ring of integers of an algebraic number field is half-factorial if
and only if its class number is at most two. Thus even the requirement
that all factorizations of each element have the same length detects a
restriction on the class group. From the late 1980s onwards, the study
of non-unique factorizations developed systematically in commutative
monoids and integral domains. For an account of this development and the relation between
arithmetical invariants and algebraic structure, see the survey \cite{Ge-Zh20Survey}.

A central invariant in this programme is the system of length sets.
Let $H$ be an atomic monoid. For a nonunit $a\in H$, the set
$\mathsf L_H(a)$ consists of the lengths of all factorizations of
$a$ into atoms. The family $\mathcal L(H)=\{\mathsf L_H(a)\colon a\in H\}$
is the system of length sets of $H$; for units, we set
$\mathsf L_H(a)=\{0\}$. Systems of length sets have been studied in
many classes of monoids and domains, including numerical and Puiseux
monoids, power monoids, monoids of ideals, rings of integer-valued
polynomials, and noetherian prime rings
\cite{Go19a,Mo-ON25,Re26a,Bo-Co-Kh26a,Fa-Wi24a,B-B-N-S23a}.
We focus on transfer Krull monoids over abelian groups. This class
includes commutative Krull domains with class group $G$ in which
every class contains a prime divisor. If $H$ is transfer Krull over
$G$, then $\mathcal L(H)=\mathcal L(\mathcal B(G))$, where
$\mathcal B(G)$ is the monoid of zero-sum sequences over $G$.
We write $\mathcal L(G):=\mathcal L(\mathcal B(G))$.
Rings of integers of algebraic number fields are classical examples,
with $G$ their ideal class group \cite{Ge-Zh20Survey}.

The converse question—namely, to what extent arithmetic properties characterize the class group—dates back to a problem posed by Narkiewicz in the 1970s (\cite[Problem 32]{Na74}). After early contributions in the 1980s by Kaczorowski, Halter-Koch, Chapman (see \cite[Chapter 9]{GGZ2026} for more on the early development),  research focused on the so-called characterization problem, which runs as follows: for finite abelian groups $G,G'$
with $\mathsf D(G)\ge4$, does
\[
 \mathcal L(G)=\mathcal L(G')\quad\Longrightarrow\quad G\cong G'
\]
hold? Here $\mathsf D(G)$ is the Davenport constant of $G$.
Length sets record only the possible numbers of irreducible factors,
yet their full system is conjectured to determine the class group.
The restriction on the Davenport constant is necessary, since
$\mathcal L(C_3)=\mathcal L(C_2^2)$. The finiteness hypothesis is
also essential. If $G$ is infinite, then every finite nonempty subset
of $\mathbb N_{\ge2}$ occurs as a length set, and
\[
 \mathcal L(G)=\big\{\{0\},\{1\}\big\}
 \cup\{L\subset\mathbb N_{\ge2}:L\text{ is finite and nonempty}\}
\]
by Kainrath's realization theorem \cite{Ka99a}. Thus all infinite
abelian groups have the same system of length sets. For further
background on the Characterization Problem, see
\cite{Sc09b} and \cite[Problem C]{Ge-Zh20Survey}.

Earlier progress is surveyed in \cite{Sc09b}. Schmid established the
characterization of $C_n^2$ in \cite{Sc09c}, and Geroldinger and Schmid
treated groups of rank two in \cite{Ge-Sc16a}. Products of two atoms
provide another source of characterizing information
\cite{Ba-Ge-Gr-Ph13a}. For homocyclic groups $C_n^r$, that is, direct
sums of $r$ cyclic groups of the same order $n$, the work of
Geroldinger and Zhong \cite{Ge-Zh16b} and Zhong
\cite{ZhongDistances,Zhong} developed methods based on minimal
distances and refined elasticities. In particular, Zhong proved the
characterization for $r\le n-3$, and for $r\ge n-1$ when $n$ is a
prime power \cite[Theorem 1.2]{Zhong}. Our main result settles the
remaining cases.

\begin{theorem}\label{t2}
Let $n\ge2$ and $r\ge1$ be integers, and let $G$ be a finite abelian
group with $\mathsf D(G)\ge4$. If
$\mathcal L(G)=\mathcal L(C_n^r)$, then $G\cong C_n^r$.
\end{theorem}

Theorem \ref{t2} settles the remaining cases $r=n-2$, and $r\ge n-1$ with $n$ not a prime power. A second main result (Theorem \ref{t1}) is an exact formula for the cross-number invariant $\mathsf K(G,r-1)$ (see Section \ref{s2}) for finite abelian groups $G$ of rank $r\ge\max\{2,\exp(G)-2\}$. The cross number of a sequence is the sum of the reciprocals of the orders of its terms, and $\mathsf K(G,d)$ is the largest cross number of a minimal zero-sum sequence over a non-half-factorial subset whose minimum distance is divisible by $d$. Our formula shows that $\mathsf K(G,r-1)$ depends only on the rank and the order of the smallest cyclic factor in the decomposition of $G$. For homocyclic groups, it settles both assertions of \cite[Conjecture 5.1]{Zhong}. Through comparisons with refined elasticities, the formula provides a key ingredient in the proof of Theorem \ref{t2}. Its precise statement and proof are given in Section \ref{s3}. For related work on transfer Krull monoids and elasticity, see \cite{Ge-Sc-Zh17b,Ge-Zh17a,GZInvariants}.

The paper is organized as follows. Section \ref{s2}
introduces the notation and collects the preliminary results.
Section \ref{s3} establishes the cross-number formula in
Theorem \ref{t1}. Sections \ref{s4}
and \ref{s5} develop the structural results and elasticity
comparisons used in Section \ref{s6} to prove
Theorem \ref{t2}. Finally,
Section \ref{s7} presents further questions.

\section{Background on zero-sum sequences and length sets}\label{s2}
We fix our notation and recall the definitions needed in the sequel.
For general background on zero-sum sequences and factorization theory,
we refer to \cite{Ge-HK06a,Ge-Ru09,GGZ2026}.

We denote by $\mathbb N$ the set of positive integers and set
$\mathbb N_0=\mathbb N\cup\{0\}$. For integers $a\le b$, we write
$[a,b]=\{k\in\mathbb Z:a\le k\le b\}$. For a prime $p$ and
$m\in\mathbb N$, let $\mathsf v_p(m)$ denote the $p$-adic valuation
of $m$.

Throughout, groups are written additively. Let $G$ be a nontrivial
finite abelian group, and write
\[
 G\cong C_{n_1}\oplus\cdots\oplus C_{n_r},
 \qquad \text{ with }\quad 1<n_1\mid\cdots\mid n_r,
\]
where $C_m$ denotes a cyclic group of order $m$. We denote by
$\mathsf r(G)=r$ and $\exp(G)=n_r$ the rank and exponent of $G$,
respectively. For $g\in G$, let $\ord(g)$ denote its order; in
particular, $\ord(0)=1$. For a prime $p$, we denote by $G_p$ the
$p$-primary component of $G$ and by $\pi_p:G\to G_p$ the canonical
projection.
For an integer $m$ and a subset $A\subset G$, we set
$m\cdot A=\{ma:a\in A\}$.

Let $G_0\subset G$. We denote by $\mathcal F(G_0)$ the
multiplicatively written free abelian monoid with basis $G_0$.
Its elements are finite unordered sequences over $G_0$, with
repetitions allowed. For a sequence
\[
 S=g_1\cdot\ldots\cdot g_{\ell}
   =\prod_{g\in G_0}g^{\mathsf v_g(S)}\in\mathcal F(G_0),
\]
we denote by $\mathsf v_g(S)$ the multiplicity of $g$ in $S$ and set
\[
 |S|=\ell,\qquad
 \supp(S)=\{g\in G_0:\mathsf v_g(S)>0\},\qquad
 \sigma(S)=g_1+\ldots+g_{\ell}=\sum_{g\in G_0}\mathsf v_g(S)g.
\]
The identity element $1=1_{\mathcal F(G_0)}$ is the trivial sequence with $|1|=0$ and
$\sigma(1)=0$. We call a sequence \emph{squarefree} if each of its
terms has multiplicity one. We write $T\mid S$
if $T$ is a subsequence of $S$. In this case, we write $ST^{-1}$ for the
sequence obtained by deleting $T$ from $S$. We call $S$ a zero-sum
sequence if $\sigma(S)=0$, and zero-sum free if it has no nontrivial
zero-sum subsequence. The monoid of zero-sum sequences over $G_0$ is
\[
 \mathcal B(G_0)=\{S\in\mathcal F(G_0):\sigma(S)=0\}.
\]
Its atoms (or irreducibles) are the minimal zero-sum sequences,
that is, the nontrivial zero-sum sequences having no proper
nontrivial zero-sum subsequence.
We set $\mathcal A(G_0):=\mathcal A(\mathcal B(G_0))$ and denote by
$\mathsf D(G)=\max\{|U|:U\in\mathcal A(G)\}$ the Davenport constant of $G$.

For $B\in\mathcal B(G_0)$, we denote by $\mathsf L(B)$ its length set.
Thus $k\in\mathsf L(B)$ if and only if $B$ has a factorization into $k$
atoms, and $\mathsf L(1)=\{0\}$ by convention. We set
\[
 \mathcal L(G_0):=\mathcal L(\mathcal B(G_0))=\{\mathsf L(B):B\in\mathcal B(G_0)\}.
\]
The following lemma recalls the connection between length sets in
Krull monoids and those in monoids of zero-sum sequences. For
background on Krull monoids and their class groups, see
\cite{Ge-HK06a,GGZ2026,HK98}.

\begin{lemma}\label{lt}
Let $H$ be a Krull monoid with class group $G$, and let $G_0\subset G$
be the set of classes containing prime divisors. There is a transfer
homomorphism $\boldsymbol\beta\colon H\to\mathcal B(G_0)$ such that
\[
 \mathsf L_H(a)=\mathsf L_{\mathcal B(G_0)}(\boldsymbol\beta(a))
 \qquad\text{for every }a\in H.
\]
In particular, $\mathcal L(H)=\mathcal L(G_0)$.
\end{lemma}
\begin{proof}
See \cite[Theorem 3.4.10]{Ge-HK06a}.
\end{proof}

An integral domain $D$ is a Krull domain if and only if its
multiplicative monoid $D\setminus\{0\}$ is a Krull monoid. Thus
Theorem \ref{t2} applies to Krull domains with homocyclic class group
and prime divisors in all classes. More generally, a transfer Krull
monoid admits a weak transfer homomorphism to a monoid of zero-sum
sequences, and such a homomorphism also preserves length sets.
Thus the same connection applies to further monoids and domains
which need not be Krull. Examples and precise definitions are given
in \cite{Ge-Zh20Survey,GGZ2026}.

We call $G_0$ half-factorial if every length set in $\mathcal L(G_0)$
is a singleton. For a finite nonempty set $L\subset \mathbb N$, we set the elasticity 
$\rho(L)=\max L/\min L$, and set $\rho(\{0\})=1$. For $B\in\mathcal B(G_0)$, its
elasticity is $\rho(\mathsf L(B))$. For $G_0\ne\emptyset$,
we set
\[
 \rho(G_0)=\sup\left\{\rho(\mathsf L(B))\colon 
                  B\in\mathcal B(G_0)\right\},
\]
and set $\rho(\emptyset)=1$. For the empty subset, we have
$\mathcal F(\emptyset)=\mathcal B(\emptyset)=\{1\}$,
$\mathcal L(\emptyset)=\{\{0\}\}$, and
$\Delta(\emptyset)=\emptyset$.

For $S\in\mathcal F(G_0)$, its cross number is
\[
 \mathsf k(S)=\sum_{g\in G_0}\frac{\mathsf v_g(S)}{\ord(g)}.
\]
For $G_0\ne\emptyset$, we set
\[
 \mathsf K(G_0)=\max\{\mathsf k(U):U\in\mathcal A(G_0)\}.
\]
We call $G_0$ a \emph{large cross number set} (an \emph{LCN-set})
if every atom over $G_0$ has cross number at least one. For cross-number invariants and half-factorial
subsets, see \cite{Kr-Za91a,He14a,Kim,Ga-Wa12a,Kr13a,Pl-Sc05b}.

For a finite set $L\subset\mathbb Z$, let $\Delta(L)$ denote the
set of positive differences between consecutive elements of $L$.
We set
\[
 \Delta(G_0):=\Delta(\mathcal B(G_0))=\bigcup_{B\in\mathcal B(G_0)}\Delta(\mathsf L(B))
\]
and
\[
 \Delta^*(G)=\{\min\Delta(G_0):G_0\subset G,\ \Delta(G_0)\ne\emptyset\}.
\]
We call $\Delta(G_0)$ the set of distances of $G_0$ and
$\Delta^*(G)$ the set of minimal distances of $G$. If
$\Delta(G_0)\ne\emptyset$, then $\min\Delta(G_0)$ is the
minimum distance of $G_0$. In particular, $G_0$ is half-factorial
if and only if $\Delta(G_0)=\emptyset$. We denote by $\mathsf m(G)$ the largest
minimum distance of a non-half-factorial LCN-subset of $G$, and set
$\mathsf m(G)=0$ if no such subset exists. For results on distances
and minimal distances, see
\cite{Sc05d,Ge-Zh15a,ZhongDistances}.

Let $d,\ell\in\mathbb N$ and $M\in\mathbb N_0$.
A finite nonempty set $L\subset\mathbb Z$ is an
\emph{almost arithmetic progression} (AAP) with difference $d$,
length $\ell$, and bound $M$ if
\[
 L=y+\big(L'\cup d\cdot[0,\ell-1]\cup L''\big)
 \subset y+d\mathbb Z,
\]
where $y\in\mathbb Z$, $L'\subset[-M,-1]$, and
$L''\subset(\ell-1)d+[1,M]$. Thus its central part is an
$\ell$-term arithmetic progression, and the initial and end parts
are confined to intervals of length $M$.
For $|G|\ge3$, let $\Delta_1(G)$ be the set of all
$d\in\mathbb N$ with the following property:
\begin{itemize}
\item[] For every $k\ge1$, there exists $L\in\mathcal L(G)$ which
is an AAP with difference $d$ and length $\ell\ge k$.
\end{itemize}
The Structure Theorem for Length Sets supplies a global bound for
these AAPs; see \cite[Chapter 7]{GGZ2026}.
For $d\in\Delta_1(G)$, define
\[
 \mathsf K(G,d)=\max\{\mathsf K(G_0):G_0\subset G,
            \ \Delta(G_0)\ne\emptyset,
            \ d\mid\min\Delta(G_0)\}
\]
and
\[
 \rho^*(G,d)=\max\{\rho(G_0):G_0\subset G,
            \ \Delta(G_0)\ne\emptyset,
            \ d\mid\min\Delta(G_0)\}.
\]
These maxima are taken over nonempty families
\cite[Definition 3.3]{Zhong}. Thus $\rho^*(G,d)$ is the largest
elasticity of a non-half-factorial subset whose minimum distance is
divisible by $d$; $\mathsf K(G,d)$ is the corresponding maximum of the
cross numbers of atoms over such subsets.

We collect the results used in the proofs below.

\begin{lemma}\label{l1}
Let $G\cong C_{n_1}\oplus\cdots\oplus C_{n_r}$, where
$1<n_1\mid\cdots\mid n_r$.
\begin{enumerate}
\item If $|G|\ge3$, then $\rho(G)=\mathsf D(G)/2$ and $\mathsf D(G)\ge 1+\sum_{i=1}^r(n_i-1)$.
\item If $G$ is a $p$-group, then
$\mathsf D(G)=1+\sum_{i=1}^r(n_i-1)$ and
$\mathsf K(G)=1/\exp(G)+\sum_{i=1}^r(1-1/n_i)$.

\item A nonempty subset $G_0\subset G$ is half-factorial if and only if
$\mathsf k(U)=1$ for every $U\in\mathcal A(G_0)$.
\end{enumerate}
\end{lemma}
\begin{proof}
See \cite[Sections 5 and 6]{Ge-HK06a} and
\cite[Lemmas 2.1 and 2.2]{Zhong}.
\end{proof}

\begin{lemma}\label{l2}
Let $G$ be a finite abelian group with $|G|\ge3$.
\begin{enumerate}
\item For every non-half-factorial subset $G_0\subset G$, we have
\begin{equation*}
 \min\Delta(G_0)=\gcd\Delta(G_0).
\end{equation*}

\item We have $\Delta^*(G)\subset\Delta_1(G)\subset \{d\in \Delta(G)\colon \text{ there exists }d^*\in \Delta^*(G) \text{ such that }d\mid d^*\}$ and
\[
 \max\Delta_1(G)=\max\Delta^*(G)
 =\max\{\exp(G)-2,\mathsf r(G)-1\}.
\]
Moreover,
\[
 \mathsf m(G)\le\max\{\mathsf r(G)-1,\lfloor\exp(G)/2\rfloor-1\}.
\]
\item If $e_1,\ldots,e_s$ are independent elements of the same
order $m\ge2$, where $s\ge2$, then
\[
 \min\Delta(\{e_1,\ldots,e_s,e_1+\cdots+e_s\})=s-1.
\]
In particular, if $\mathsf r(G)\ge 2$, then $[1, \mathsf r(G)-1]\subset \Delta^*(G)$.
\end{enumerate}
\end{lemma}
\begin{proof}
See 
\cite[Propositions 1.4.4 and 6.8.2]{Ge-HK06a} for Items 1 and 3.

2. The inclusion $\Delta^*(G)\subset\Delta_1(G)$ follows from \cite[Corollary 4.3.16]{Ge-HK06a}  and the other assertions follow from \cite[Theorem 1.1]{Ge-Zh15a} and
\cite[Proposition 3.7]{ZhongDistances}.
\end{proof}

\begin{lemma}\label{l3}
Suppose $G$ and $G'$ are finite abelian groups with
$\mathsf D(G)\ge4$ and $\mathcal L(G)=\mathcal L(G')$.
Then $\mathsf D(G)=\mathsf D(G')$,
$\Delta_1(G)=\Delta_1(G')$, and
$\rho^*(G,d)=\rho^*(G',d)$ for every $d\in\Delta_1(G)$.
In particular,
\[
 \max\{\exp(G)-2,\mathsf r(G)-1\}
 =\max\{\exp(G')-2,\mathsf r(G')-1\}.
\]
If either group is isomorphic to a subgroup of the other, then
$G\cong G'$.
\end{lemma}
\begin{proof} The assertion $\mathsf D(G)=\mathsf D(G')$ follows from \cite[Proposition 7.3.1.3]{Ge-HK06a} and the other assertions follow from 
 \cite[Proposition 3.5 and Lemma 4.1]{Zhong}.
\end{proof}

\begin{lemma}\label{l4}
Let $G$ be a finite abelian group with $\exp(G)\ge3$, and let
$G_0\subset G$ be non-half-factorial with $\min\Delta(G_0)\ge \lfloor \exp(G)/2\rfloor$.
\begin{enumerate}
\item An atom $U\in \mathcal A(G_0)$ with $\mathsf k(U)<1$ is squarefree, all its terms
have order $\exp(G)$, and $|U|=\exp(G)-\min\Delta(G_0)$.
\item If $U\in \mathcal A(G_0)$ is an atom with $\mathsf k(U)\ge1$, then its support is an
LCN-set. If its support is non-half-factorial, then
$\min\Delta(\supp(U))\le|\supp(U)|-2$.
\item If $d\in\mathbb N$ divides $\min\Delta(G_0)$, then every LCN-subset of $G_0$ with at most
$d+1$ elements is half-factorial.
\end{enumerate}
\end{lemma}
\begin{proof}
Items 1 and 2 follow from
\cite[Lemma 2.6.1 and 2]{Zhong}. 

3. Suppose such a subset
$E$ were non-half-factorial. By Lemma \ref{l1}.3,
there exists an atom $U\in \mathcal A(E)$ with $\mathsf k(U)>1$, whence $\supp(U)$ is non-half-factorial. It follows from Item 2 that  $\min\Delta(\supp(U))\le |\supp(U)|-2\le |E|-2\le (d+1)-2=d-1$. But
Lemma \ref{l2}.1 implies that $\min\Delta(\supp(U))$ is a positive
multiple of $d$, a contradiction.
\end{proof}

\begin{lemma}\label{l5}
Let $G\cong C_{n_1}\oplus\cdots\oplus C_{n_r}$ with
$1<n_1\mid\cdots\mid n_r$, and suppose
$r\ge\lfloor n_r/2\rfloor+1$.
\begin{enumerate}
\item If $d\in\Delta_1(G)$ satisfies
$\lfloor n_r/2\rfloor\le d\le r-1$, and
$G_0\subset G$ is non-half-factorial with $d\mid\min\Delta(G_0)$,
then each $U\in\mathcal A(G_0)$ with $\mathsf k(U)>1$ satisfies
\[
 \mathsf k(U)=1+\frac{sd}{n_{r-d}}\qquad\text{for some }s\in\mathbb N.
\]
\item We have
\[
 \mathsf K(G,r-1)\ge
 1+(r-1)\sum_{p\mid n_1}(1-p^{-\mathsf v_p(n_1)}).
\]
\item
If $n_r$ is a prime power, then
$\mathsf K(G,r-1)=1+(r-1)(1-1/n_1)$.
\end{enumerate}
\end{lemma}
\begin{proof}
See
\cite[Lemma 3.7]{Zhong}.
\end{proof}

\begin{lemma}\label{l6}
Let $G$ be a finite abelian group and let $d\in\Delta_1(G)$.
\begin{enumerate}
\item We have $\rho^*(G,d)\ge\mathsf K(G,d)$. If
$1\le d\le\mathsf r(G)-1$, then
$\mathsf K(G,d)\ge1+d(1-1/n_1)\ge1+d/2$,
where $n_1$ is given by
$G\cong C_{n_1}\oplus\cdots\oplus C_{n_r}$ with
$1<n_1\mid\cdots\mid n_r$.
\item If $G_0\subset G$ is an LCN-set with
$\Delta(G_0)\ne\emptyset$ and $d\mid\min\Delta(G_0)$, then
\[
 \max\mathsf L(B)\le\mathsf k(B)\le\mathsf K(G,d)\min\mathsf L(B)
 \qquad \text{ for every nontrivial }B\in\mathcal B(G_0).
\]
In particular, $\rho(G_0)\le\mathsf K(G,d)$.
\item If $d>\exp(G)-2$, every subset $G_0$ with $\Delta(G_0)\neq \emptyset$ and $d\mid \Delta(G_0)$  is an LCN-set, and
\[
 \rho^*(G,d)=\mathsf K(G,d).
\]
\item Removing $0$ from a subset changes neither its distances nor
its elasticity.
\end{enumerate}
\end{lemma}
\begin{proof}
Item 1 follows from \cite[Lemma 3.6]{Zhong}. 

2. Let $B\in \mathcal B(G_0)$. If $B$ is trivial, then the assertion is clear.  Suppose $B=U_1\cdot\ldots\cdot U_k=V_1\cdot\ldots\cdot V_{\ell}$ with $k=\min \mathsf L(B)\ge 1$ and $\ell=\max\mathsf L(B)$, where $U_1,\ldots, U_k,V_1,\ldots, V_{\ell}\in \mathcal A(G_0)$. Since $G_0$ is an LCN-set, we obtain $\mathsf k(V_i)\ge 1$ and hence $\mathsf k(B)\ge \ell$.
At the same time, we have $\mathsf k(B)=\sum_{i=1}^k \mathsf k(U_i)\le k\mathsf K(G_0)\le k\mathsf K(G,d)$. 
It follows that $\ell/k=\max\mathsf L(B)/\min \mathsf L(B)\le \mathsf K(G,d)$ for all nontrivial $B\in \mathcal B(G_0)$, whence $\rho(G_0)\le\mathsf K(G,d)$.

3. Set $n=\exp(G)$. Assume to the contrary that there exists $U\in \mathcal A(G_0)$ with $\mathsf k(U)<1$. 
Then $\{n, n\mathsf k(U)\}\subset \mathsf L(U^n)$
implies that $d$ divides $n-n\mathsf k(U)$. Since $|U|\ge 2$, we obtain $\mathsf k(U)\ge 2/n$ and hence $1\le n-n\mathsf k(U)\le n-2$, a contradiction to $d>n-2$. Thus $G_0$ is an LCN-set. Now the equality follows from Items 1 and 2.

4. Let $B=0^{v}B_0$ be a zero-sum sequence with $v\in \mathbb N_0$ and $0\nmid B_0$. Then $\mathsf L(B)=v+\mathsf L(B_0)$, whence $\Delta(\mathsf L(B))=\Delta(\mathsf L(B_0))$ and $\rho(\mathsf L(B))\le \rho(\mathsf L(B_0))$. Now the assertion follows.
\end{proof}

\begin{lemma}\label{l7}
Let $G\cong C_{n_1}\oplus\cdots\oplus C_{n_r}$ with
$1<n_1\mid\cdots\mid n_r$, and let $d\in\Delta_1(G)$ satisfy
$d\ge\max\{\lfloor n_r/2\rfloor,\lfloor(r+1)/2\rfloor\}$.
\begin{enumerate}
\item If $d\ge r$, then $\rho^*(G,d)=n_r/(n_r-d)$.
\item Suppose $r=n_r-1\ge3$. If $d=r-1$, then
$\rho^*(G,d)=\mathsf K(G,d)$ unless $n_1=n_r$ is a prime power.
In that exceptional case, we have
$\rho^*(G,r-1)=1+(r+1)(r-1)/(r+2)$.
\end{enumerate}
\end{lemma}
\begin{proof}
See \cite[Proposition 3.8]{Zhong}.
\end{proof}

\begin{lemma}\label{l8}
Let $H=C_n^r$ with $n\ge2$ and $r\ge1$. Suppose that $\mathsf D(H)\ge4$ and that $G$ is a finite abelian group with
$\mathcal L(G)=\mathcal L(H)$.
\begin{enumerate}
\item If $r\le n-3$, then $G\cong H$.

\item If $r\ge n-1$ and $n$ is a prime
power, then $G\cong H$.
\item Suppose $r\ge n-1$. If
\[
 \mathsf K(H,r-1)=1+\frac{s(r-1)}n
 \quad\text{for some }s\in\mathbb N\text{ with }\gcd(s,n)=1,
\]
then $G\cong H$.
\end{enumerate}
\end{lemma}
\begin{proof}
See \cite[Theorem 1.2 and Proposition 5.2]{Zhong}.
\end{proof}

\section{Cross numbers with respect to minimal distances}\label{s3}
Cross numbers have a history closely tied to these arithmetical
questions. Introduced by Krause in 1984 in the study of algebraic
number fields \cite{Kr84}, and developed by Krause and Zahlten for
Krull monoids \cite{Kr-Za91a}, the cross number of a sequence
$S=g_1\cdot\ldots\cdot g_{\ell}$ is
\[
 \mathsf k(S)=\sum_{i=1}^{\ell}\frac{1}{\ord(g_i)}.
\]
The Davenport constant is the maximum length of a minimal zero-sum
sequence. The maximal cross number is the corresponding weighted
maximum, with each term $g$ assigned weight $1/\ord(g)$.
The cross-number criterion for half-factoriality gives a direct
connection with factorization theory: a subset
$G_0$ of a finite abelian group is half-factorial if and only if every
atom over $G_0$ has cross number one \cite{Ge-HK06a}.
More generally, cross numbers allow one to compare factorization
lengths by using an additive weight on zero-sum sequences.
Determining the structure of minimal zero-sum sequences attaining the
maximal cross number is an associated inverse problem in additive
combinatorics
\cite{Ge-Gr09Cross}.

Determining the maximal cross number $\mathsf K(G)$ is a difficult
extremal problem. If $G\cong\bigoplus_{i=1}^t C_{q_i}$, where each
$q_i$ is a prime power, the classical cross-number conjecture predicts
\[
 \mathsf K(G)=\frac{1}{\exp(G)}+
             \sum_{i=1}^t\left(1-\frac{1}{q_i}\right).
\]
The right-hand side is attained by an explicit construction; the difficulty
is to prove the matching upper bound for arbitrary minimal zero-sum
sequences \cite{Kr-Za91a,Kim}. The formula holds for $p$-groups
\cite[Section 5.5]{Ge-HK06a}, while subsequent work has treated further families
and developed bounds and structural methods \cite{Ge-Gr09Cross,He14a,Kim}.
For groups with several prime divisors, projecting a minimal zero-sum
sequence to its primary components need not preserve minimality,
which obstructs a direct reduction to the $p$-group case.
The invariant considered here restricts the cross-number problem to
subsets satisfying a condition on their minimum distance. Its exact
value can then be used, through refined elasticities, to recover
information about the cyclic factors of the group.

We study the invariant $\mathsf K(G,d)$ introduced in
\cite[Definition 3.3]{Zhong}. It is the largest cross number of an atom
over a subset whose set of distances is nonempty and has minimum
divisible by $d$. Precise definitions are given in Section \ref{s2}.
The following theorem determines this invariant at $d=\mathsf r(G)-1$
when $\mathsf r(G)\ge\max\{2,\exp(G)-2\}$. The formula shows that, in
this range, the invariant depends only on the rank $r$ and the order
$n_1$ of the smallest cyclic factor in the displayed decomposition.

\begin{theorem}\label{t1}
Let $G\cong C_{n_1}\oplus\cdots\oplus C_{n_r}$ be a finite abelian
group with $1<n_1\mid\cdots\mid n_r$ and
$r\ge\max\{2,n_r-2\}$. Then
\begin{equation*}
 \mathsf K(G,r-1)=1+(r-1)\sum_{p\mid n_1}
 \left(1-\frac1{p^{\mathsf v_p(n_1)}}\right),
\end{equation*}
where the sum runs over the distinct prime divisors of $n_1$.
\end{theorem}

For $G=C_n^r$ and $r\ge n-1$, Theorem \ref{t1}
proves Conjecture 5.1, C2, of \cite{Zhong}. It also proves C1:
letting $s=n\sum_{p\mid n}(1-p^{-\mathsf v_p(n)})$, one has $\gcd(s,n)=1$.
Indeed, for each prime $p\mid n$, the integer $s$ is congruent to
$-n/p^{\mathsf v_p(n)}$ modulo $p$.

We prove Theorem \ref{t1} in this section. The main task is to obtain
the upper bound for $\mathsf K(G,r-1)$. We first determine how concatenating
atoms of the form $g^{\ord(g)}$ affects the maximum of factorization
lengths.
\begin{lemma}
\label{l9}
Let $G$ be a finite abelian group with exponent $n\ge3$, and let
$G_0\subset G$ be an LCN-set with $\min\Delta(G_0)\ge n-3$.
For every $A\in\mathcal A(G_0)$ and every family
$(c_g)_{g\in\supp(A)}$ of nonnegative integers, we have
\[
 \max\mathsf L\left(A\prod_{g\in\supp(A)}\left(g^{\ord(g)}\right)^{c_g}\right)
 =1+\sum_{g\in\supp(A)}c_g.
\]
\end{lemma}

\begin{proof}
Let $B=\prod_{g\in\supp(A)}(g^{\ord(g)})^{c_g}$.
Set $d=\min\Delta(G_0)$ and $C=\sum_{g\in \supp(A)} c_g$. Then $C+1\in \mathsf L(AB)$. If $\mathsf k(A)=1$, the
LCN property gives $\max\mathsf L(AB)\le\mathsf k(AB)=C+1$.

Suppose $\mathsf k(A)>1$ and choose a counterexample with $C$ minimal.
Then $C\ge1$, $\mathsf v_g(A)<\ord(g)$ for every $g\in\supp(A)$, and
there is a factorization $AB=A_1\cdot\ldots\cdot A_{\ell}$ with
\begin{equation}\label{e3}
 \ell\ge C+d+1.
\end{equation}

If there exists $i\in [1, \ell]$, say $i=1$, such that $A_1=(g^*)^{\ord(g^*)}$ for some $g^*\in \supp(A)$, then $\ord(g^*)\le \mathsf v_{g^*}(AB)=\mathsf v_{g^*}(A)+c_{g^*}\ord(g^*)$ and hence $c_{g^*}\ge 1$. Letting $$B^*=\big((g^*)^{\ord(g^*)}\big)^{-1}B=\big((g^*)^{\ord(g^*)}\big)^{c_{g^*}-1}\prod_{g\in\supp(A)\setminus\{g^*\}}\left(g^{\ord(g)}\right)^{c_g}\,,$$
it follows that 
\[
AB^*=A_2\cdot\ldots\cdot A_{\ell}\,,
\]
a contradiction to the minimality of $C$. Thus 
\begin{equation}\label{eq-supp}
	\text{$|\supp(A_i)|\ge 2$ for all $i\in [1,\ell]$.}
\end{equation}

If there exists $i\in [1, \ell]$, say $i=1$, such that $A_1=A$, then $B$ would have a factorization of length $\ell-1>C=\mathsf k(B)$, contrary to the LCN property. Thus $A_i\neq A$ for all $i\in [1,\ell]$.

Choose $g^*$ with $c_{g^*}>0$, and let $I\subset[1,\ell]$ satisfy
$(g^*)^{\ord(g^*)}\mid\prod_{i\in I}A_i$.  Then $R:=\left((g^*)^{\ord(g^*)}\right)^{-1}\prod_{i\in I}A_i$
is nontrivial by \eqref{eq-supp}. Since $AB\left((g^*)^{\ord(g^*)}\right)^{-1}=R\prod_{j\in [1,\ell]\setminus I}A_j$, it follows from the minimality of $C$ that 
$\max\mathsf L(R)+\ell-|I|\le 1+C-1=C$. Combined with \eqref{e3}, we obtain that 
\begin{equation}\label{e4}
 |I|\ge \max \mathsf L(R)+\ell-C\ge 1+d+1=d+2\,.
\end{equation}
Since $d\ge\max\{1,n-3\}$ and $n\ge3$, we have
$d+2>\lceil n/2\rceil$.

Let $J_{g^*}=\{i\in [1,\ell]:\mathsf v_{g^*}(A_i)\ge2\}$ and
$E_{g^*}=\sum_{i\in J_{g^*}}(\mathsf v_{g^*}(A_i)-1)$.
If $\sum_{i\in J_{g^*}}\mathsf v_{g^*}(A_i)\ge\ord(g^*)$, then we may choose $I\subset J_{g^*}$ with $|I|\le \lceil\ord(g^*)/2\rceil$,
contradicting \eqref{e4}.
Otherwise, since $\mathsf v_{g^*}(AB)>\ord(g^*)$, there exists $I\subset [1,\ell]$ with $J_{g^*}\subset I$ and $|I|=\ord(g^*)-E_{g^*}$ such that $\sum_{i\in I}\mathsf v_{g^*}(A_i)=\ord(g^*)$.
It follows from   $d\ge n-3$ and \eqref{e4} that 
\[
 E_{g^*}=\ord(g^*)-|I|\le\ord(g^*)-d-2\le n-d-2\le1.
\]
Set 
\[
\delta_{g, i}=
\begin{cases}
	\mathsf v_g(A_i)-1, \quad \quad &\text{if }\mathsf v_g(A_i)\ge 2,\\
     0,                  &\text{otherwise}.
\end{cases}
\]
Then $E_{g^*}=\sum_{i=1}^{\ell}\delta_{g^*,i}$ and we have shown that for every $g\in \supp(A)$ with $c_g>0$, the following holds.
\[
\sum_{i=1}^{\ell}\delta_{g,i}\le 1\,.\]

At the same time, for every $i\in [1,\ell]$, there exists some $g_i\in \supp(A)$ such that $\mathsf v_{g_i}(A_i)>\mathsf v_{g_i}(A)\ge 1$, whence $c_{g_i}>0$ and $\mathsf v_{g_i}(A_i)\ge 2$. It follows that 
\[
\sum_{g\in \supp(A) \text{ with }c_g>0}\delta_{g,i}\ge 1\,.
\]
Therefore
\[
 \ell=\sum_{i=1}^{\ell}1\le \sum_{i=1}^{\ell}\sum_{g\in \supp(A) \text{ with }c_g>0}\delta_{g,i}=\sum_{g\in \supp(A) \text{ with }c_g>0}\sum_{i=1}^{\ell}\delta_{g,i}\le|\{g\in \supp(A)\colon c_g>0\}|\le C,
\]
contrary to \eqref{e3}.
\end{proof}


\begin{proposition}\label{p1}
Let $G\cong C_{n_1}\oplus\cdots\oplus C_{n_r}$ be a finite abelian
group with $1<n_1\mid\cdots\mid n_r$, $n_r\ge6$, and
$r\ge n_r-2$. Let $G_0\subset G$ be an LCN-set with
$\min\Delta(G_0)=r-1$, and let $p\mid n_r$ be prime.
Suppose $B\in\mathcal B(G_0)$ has the form
\[
 B=\prod_{\substack{g\in G_0\\p\mid\ord(g)}}
 g^{b_g\ord(g)/p^{\mathsf v_p(\ord(g))}},\qquad \text{ where }\quad b_g\in\mathbb N_0,
\]
and suppose $\beta$ is a real number with $0\le \beta<1$ such that 
\[
\beta\equiv (\mathsf k(B)-\ell)/(r-1) \pmod {\mathbb Z}
\]  for some
$\ell\in\mathsf L(B)$. Then 
\begin{equation*}
 p^{\mathsf v_p(n_1)}\beta\in\mathbb Z \quad  \text{ and }\quad \max\mathsf L(B)=\mathsf k(B)-(r-1)\beta.
\end{equation*}
\end{proposition}

\begin{proof}
By Lemma \ref{l2}.1, all factorization lengths of a given zero-sum sequence are congruent
modulo $r-1$. Then for every $D\in \mathcal B(G_0)$, 
\begin{equation}\label{ee3}
\text{the residue of $(\mathsf k(D)-\ell)/(r-1)$ modulo
	$\mathbb Z$ is independent of $\ell\in\mathsf L(D)$},	
\end{equation}
whence $\beta$ is uniquely determined. Since $G_0$ is an LCN-set, we have $\mathsf k(B)\ge \max\mathsf L(B)$, whence
\[
\frac{\mathsf k(B)-\max\mathsf L(B)}{r-1}\ge \beta\,.
\]
 It follows that 
\begin{equation}\label{ee11}
	\max\mathsf L(B)\le \mathsf k(B)-(r-1)\beta\,.
\end{equation}

Since $r\ge n_r-2$ and $n_r\ge6$, we have
$r\ge\lfloor n_r/2\rfloor+1$ and $r-1\ge\lfloor n_r/2\rfloor$.
Note that
$r-1\in\Delta^*(G)\subset\Delta_1(G)$ by Lemma \ref{l2}.3.
We may apply Lemma \ref{l5}.1. Since $G_0$ is an LCN-set,
every atom $U$ over $G_0$ satisfies
$\mathsf k(U)=1+(r-1)s_U/n_1$ for some $s_U\in\mathbb N_0$,
where $s_U=0$ if $\mathsf k(U)=1$.
Let $D\in \mathcal B(G_0)$ and 
suppose $D=U_1\cdot\ldots\cdot U_{\ell}$, where $U_1,\ldots, U_{\ell}\in \mathcal A(G_0)$. Then 
\begin{equation}\label{ee2}
 \frac{\mathsf k(D)-\ell}{r-1}=\frac{1}{n_1}\sum_{i=1}^{\ell}s_{U_i}
 \in\frac{1}{n_1}\mathbb Z.
\end{equation}

We now construct an auxiliary finite abelian $p$-group. Set 
\[
G_1:=\{g\in G_0\colon \ord(\pi_p(g))>1\},\qquad \lambda_g:=\ord(\pi_p(g)),\qquad \text{ and }\quad  m_g=\ord(g)/\lambda_g\,.
\]
Let $M_p\subset\mathcal F(G_0)$ be the submonoid generated by
 $\{g^{m_g}\colon g\in G_1\}$, and set $\mathcal B_p=M_p\cap\mathcal B(G_0)$.
For $U,V\in M_p$, write $U\sim V$ if there exist $C,D\in\mathcal B_p$ with
$\mathsf k(C)-\ell_C,\mathsf k(D)-\ell_D\in (r-1)\mathbb Z$ such that
$UC=VD$, where
$\ell_C\in\mathsf L(C)$ and $\ell_D\in\mathsf L(D)$. We write
$F_p:=M_p/\!\sim=\{[U]\colon U\in M_p\}$ additively, so that $[U_1]+[U_2]=[U_1U_2]\in F_p$ for all $U_1,U_2\in M_p$. Note that
\begin{equation}\label{ee5}
	\text{ $[U]=0\in F_p$ if and only if  $U\in \mathcal B_p$ with $\mathsf k(U)-\ell_U\in (r-1)\mathbb Z$ for some $\ell_U\in \mathsf L(U)$.}
\end{equation}
Since $(g^{m_g})^{\lambda_g}=g^{\ord(g)}$ is an atom of cross
number one, each $[g^{m_g}]\in F_p$
has order dividing $\lambda_g$.
Thus $F_p$ is a finite abelian $p$-group of exponent dividing $p^{\mathsf v_p(n_r)}$.

The sum map $\sigma$ on $M_p$ induces a surjection
$\overline\sigma:F_p\to\pi_p\langle G_0\rangle$ with
$\overline\sigma([U])=\sigma(U)$. 
Since $m_gg\in G_p$, $\ord(m_gg)=\lambda_g$, and $\gcd(\lambda_g, m_g)=1$  for all $g\in G_0$, we obtain that 
$\pi_p\langle G_0\rangle=\langle m_gg\colon g\in G_1\rangle$. Note that $\ker(\overline\sigma)=\{[U]\in F_p\colon U\in \mathcal B_p\}$. By \eqref{ee3}, the map
\[
 \psi\colon\ker(\overline\sigma)\longrightarrow\mathbb Q/\mathbb Z,
 \qquad \text{ defined by } [D]\longmapsto\frac{\mathsf k(D)-\ell}{r-1}+\mathbb Z,
 \quad \text{ where } \ell\in\mathsf L(D),
\]
is a well-defined homomorphism, since
equivalent sequences have the same residue. It follows from \eqref{ee5} that  $[U]=0\in F_p$ if and only if $\psi([U])=0+\mathbb Z$, whence $\psi$ is injective.

Suppose $D\in \mathcal B_p$. Then  $D=\prod_{g\in G_1}(g^{m_g})^{b_g}$ for some $b_g\in \mathbb N_0$. We have
$$D^{p^{\mathsf v_p(n_r)}}=\prod_{g\in G_1}(g^{\ord(g)})^{b_g p^{\mathsf v_p(n_r)}/\lambda_g}\,,$$ 
whence $\mathsf k(D^{p^{\mathsf v_p(n_r)}})\in \mathsf L(D^{p^{\mathsf v_p(n_r)}})$ and hence ${p^{\mathsf v_p(n_r)}}\psi([D])=\psi([D^{p^{\mathsf v_p(n_r)}}])=0+\mathbb Z$. It follows that 
 $\psi([D])$ has order dividing $p^{\mathsf v_p(n_r)}$.
As $\gcd(n_1,p^{\mathsf v_p(n_r)})=p^{\mathsf v_p(n_1)}$, we infer by \eqref{ee2} that
\begin{equation}\label{e6}
 \psi\colon\ker(\overline\sigma)\hookrightarrow\frac{1}{p^{\mathsf v_p(n_1)}}\mathbb Z/\mathbb Z,
\end{equation}
whence $\ker(\overline\sigma)$ is cyclic of order dividing $p^{\mathsf v_p(n_1)}$. Combined with
$\mathsf r(\pi_p\langle G_0\rangle)\le \mathsf r(G)=r$, we have 
\[
\text{$\mathsf r(F_p)\le r+1$\quad  and\quad  $\exp(F_p)\le p^{\mathsf v_p(n_r)}$.}
\]
If $n_r$ is not a prime power, then $p^{\mathsf v_p(n_r)}\le n_r/2<r$ (in view of $r\ge n_r-2$ and $n_r\ge 6$), whence Lemma \ref{l1}.2 implies 
\begin{equation}\label{e7}
 \mathsf K(F_p)\le\mathsf K(C_{p^{\mathsf v_p(n_r)}}^{r+1})=1+r(1-p^{-\mathsf v_p(n_r)})<r.
\end{equation}
If $n_r$ is a prime power, then $p^{\mathsf v_p(n_r)}=n_r$ and instead
\begin{equation}\label{ee1}
 \mathsf K(F_p)\le1+r(1-1/n_r)<r+1,\qquad
 \mathsf K(G)\le1+(r-1)(1-1/n_r)<r.
\end{equation}

We next lift sequences over $\{\widetilde{g}\colon g\in G_1\}\subset F_p$ back to sequences over $G_1\subset G_0$.
For $g\in G_1$, set
$\widetilde g:=[g^{m_g}]\in F_p$.
The element $\widetilde g$ has order $\lambda_g$, since 
$\lambda_g\widetilde g=[g^{m_g\lambda_g}]=[g^{\ord(g)}]=0\in {F_p}$ and  $\ord(\overline{\sigma}(\widetilde g))=\ord(m_g g)=\lambda_g$. From a sequence $\widetilde D$ over $\{\widetilde{g}\colon g\in G_1\}$, 
replace each occurrence of $\widetilde g$ by $g^{m_g}$ to obtain
a sequence $D\in M_p$,
retaining  the index $g$ of each occurrence when distinct $g$ give
the same element of $F_p$. Let $\sigma_p$ be the sum map from sequences over $F_p$ to $F_p$. Note that $\overline{\sigma}$ is a map from $F_p$ to $\pi_p\langle G_0\rangle$.
Then
\begin{equation}\label{e8}
 [D]=\sigma_p(\widetilde D),\qquad
 \sigma(D)=\overline\sigma(\sigma_p(\widetilde D)),\qquad \text{ and }\ 
 \mathsf k(D)=\mathsf k(\widetilde D).
\end{equation}
In particular, 
\begin{equation}\label{ee9}
\begin{aligned}
	&\text{$D$ is zero-sum}, &&\text{when
		$\sigma_p(\widetilde D)\in\ker(\overline\sigma)$;}\\
		\text{and} \quad &\text{$(\mathsf k(D)-\ell)/(r-1)$ is an integer for
			$\ell\in\mathsf L(D)$}, &&\text{when $\sigma_p(\widetilde D)=0\in F_p$.}
\end{aligned}
\end{equation}

We next compare the cross number of a lifted sequence with its
factorization lengths. Let $D\in \mathcal B_p$ be nontrivial
such that $\widetilde{D}$ is either an atom over $\{\widetilde{g}\colon g\in G_1\}\subset F_p$ or a nontrivial zero-sum free sequence  over $\{\widetilde{g}\colon g\in G_1\}\subset F_p$. Let $t\in \mathsf L(D)$. Since $G_0$ is an LCN-set, we obtain $t\le \mathsf k(D)$. If $n_r$ is not a prime power, then \eqref{e7} and \eqref{e8} imply that
$\mathsf k(D)=\mathsf k(\widetilde{D})\le \mathsf K(F_p)<r\le (r-1)+t$.  If $n_r$ is a prime power and $D$ is an atom, then \eqref{ee1}  implies that $\mathsf k(D)\le \mathsf K(G)<r=(r-1)+t$. If $n_r$ is a prime power and $D$ is not an atom, then \eqref{ee1} and \eqref{e8} imply that $\mathsf k(D)=\mathsf k(\widetilde{D})\le \mathsf K(F_p)<r+1\le (r-1)+t$.
Therefore
\begin{equation}\label{e10}
 0\le\mathsf k(D)-t<(r-1)\qquad \text{ for every }\ t\in\mathsf L(D).
\end{equation}
In particular, if $\sigma_p(\widetilde D)=0\in F_p$, then $(\mathsf k(D)-t)/(r-1)$ is an integer for one
$t\in\mathsf L(D)$, whence it is an integer for all such $t$. Therefore
\eqref{e10} forces $t=\mathsf k(D)$ throughout, which means that 
\begin{equation}\label{ee7}
 \text{$\widetilde{D}$ is an atom }\quad  \text{ implies }\quad \mathsf L(D)=\{\mathsf k(D)\}.
\end{equation}

Note that $B=\prod_{g\in G_1}
g^{b_gm_g}$ is zero-sum.
Let $\widetilde B=\prod_{g\in G_1}\widetilde g^{b_g}$. Then
$\gamma:=\sigma_p(\widetilde B)=[B]\in\ker(\overline\sigma)$.
We have $\psi(\gamma)=\beta+\mathbb Z$. Since $|\ker(\overline\sigma)|$ divides $p^{\mathsf v_p(n_1)}$ by \eqref{e6}, we obtain 
\begin{equation}\label{e9}
	p^{\mathsf v_p(n_1)}\beta\in\mathbb Z,\qquad \text{ and hence }\ 0\le\beta\le1-p^{-\mathsf v_p(n_1)}.
\end{equation}
By \eqref{ee11}, it remains to show
\[
\mathsf k(B)-(r-1)\beta\in \mathsf L(B).
\]

If $\gamma=0\in F_p$, then $\beta=0$. Factor $\widetilde B$ into atoms $\widetilde{D}$ over $F_p$.
For each such atom $\widetilde{D}$, let $D$ denote its lift. We infer by \eqref{ee7} that $\mathsf L(D)=\{\mathsf k(D)\}$ and hence $\mathsf k(B)\in\mathsf L(B)$. It follows that $\mathsf k(B)-(r-1)\beta\in \mathsf L(B)$.

Suppose $\gamma\ne0\in F_p$. Assume  $\widetilde B(-\gamma)=\widetilde{D_1}\cdot\ldots\cdot \widetilde{D_x}\cdot \widetilde{T}(-\gamma)$, where  $\widetilde{D_1},\ldots, \widetilde{D_x}, \widetilde{T}(-\gamma)$ are atoms over $F_p$ and $\widetilde{T}$ is zero-sum free over $F_p$.
Let $D_i$ be the lift of $\widetilde{D_i}$ and $T$ be the lift of $\widetilde{T}$. Then $B=D_1\cdot\ldots\cdot D_x\cdot T$ and \eqref{ee7} implies that $\mathsf k(D_1\cdot\ldots\cdot D_x)\in \mathsf L(T^{-1}B)$. By \eqref{ee9}, we have $\beta\equiv (\mathsf k(T)-\ell_T)/(r-1) \pmod {\mathbb Z}$ for some $\ell_T\in \mathsf L(T)$, and  
by \eqref{e10}, we have
$0\le (\mathsf k(T)-\ell_T)/(r-1)<1$, whence $\beta=(\mathsf k(T)-\ell_T)/(r-1)$
and hence  $\mathsf k(B)-(r-1)\beta=\mathsf k(T^{-1}B)+\ell_T\in \mathsf L(B)$.

\end{proof}

We now combine Proposition \ref{p1} and 
Lemma \ref{l9} to obtain an upper bound for the cross number of an atom.

\begin{proposition}\label{p2}
Let $G\cong C_{n_1}\oplus\cdots\oplus C_{n_r}$, where
$1<n_1\mid\cdots\mid n_r$ and $r\ge n_r-2\ge 4$.
Let $G_0\subset G$ be an LCN-set with
$\min\Delta(G_0)=r-1$. Then every
$A\in\mathcal A(G_0)$ with $\mathsf k(A)>1$ satisfies
\[
 \mathsf k(A)\le1+(r-1)\sum_{p\mid n_1}(1-p^{-\mathsf v_p(n_1)}).
\]
\end{proposition}

\begin{proof}
Let $p_1,\ldots,p_u$ be the distinct prime divisors of $n_r$ and 
let $\pi_j$ denote the projection onto the $p_j$-primary component of $G$.
For $g\in G_0$, set
\begin{equation}\label{e-p-2}
\text{$\lambda_{j,g}=\ord(\pi_j(g))$
\quad 	and\quad  $m_{j,g}=\ord(g)/\lambda_{j,g}$.}
\end{equation} Since $\lambda_{j,g}$ is a power of $p_j$ and $p_j\nmid m_{j,g}$, we have $\gcd(\lambda_{j,g}, m_{j,g})=1$ for all $g\in G_0$ and $j\in [1,u]$. Furthermore, we have
\begin{equation}\label{e-p-1}
\lambda_{k,g}\mid m_{j,g} \text{ for any two distinct }j,k\in [1,u]\,.
\end{equation}

Let $A\in\mathcal A(G_0)$ with $\mathsf k(A)>1$. Then
$0\notin\supp(A)$.  For $\lambda_{j,g}>1$,
 there is
a unique $b_{j,g}\in [0, \lambda_{j,g}-1]$ 
 such that
\begin{equation}\label{e-p-3}
	b_{j,g}m_{j,g}\equiv \mathsf v_g(A)\pmod{\lambda_{j,g}}\,.
\end{equation}
For $\lambda_{j,g}=1$, set $b_{j,g}=0$, and let
$$A_j=\prod_{g\in G_0}g^{b_{j,g}m_{j,g}}.$$
It follows from \eqref{e-p-3} and \eqref{e-p-2} that
$b_{j,g}m_{j,g}g=\pi_j(\mathsf v_g(A) g)$, and hence
\[
 \sigma(A_j)=\pi_j(\sigma(A))=0.
\]
Thus $A_j\in\mathcal B(G_0)$ has the form required in
Proposition \ref{p1}.
For each $g$, we infer from \eqref{e-p-1} that 
$b_{k,g}m_{k,g}\equiv0\pmod{\lambda_{j,g}}$ for distinct $k,j\in [1,u]$, whence \eqref{e-p-3} implies 
\[
\text{$\sum_{k=1}^u b_{k,g}m_{k,g}\equiv \mathsf v_g(A) \pmod{\lambda_{j,g}}$\quad  for every $j\in [1, u]$.}
\] 
Since $\prod_{j=1}^u\lambda_{j,g}=\ord(g)$, the Chinese remainder theorem yields
\[
 \sum_{k=1}^u b_{k,g}m_{k,g}=\mathsf v_g(A)+\kappa_g\ord(g),\qquad \text{ for some }\kappa_g\in\mathbb N_0.
\]
 Thus
\begin{equation}\label{e11}
 \prod_{j=1}^u A_j=AP,\qquad \text{ with }
 P:=\prod_{g\in\supp(A)}(g^{\ord(g)})^{\kappa_g}\qquad \text{ and }\quad 
 C:=\sum_{g\in\supp(A)}\kappa_g.
\end{equation}
For each $j$, let $\beta_j$ be real with $0\le \beta_j<1$ such that $\beta_j$ is congruent to
$(\mathsf k(A_j)-\ell_j)/(r-1)$ modulo $\mathbb Z$, where
$\ell_j\in\mathsf L(A_j)$.
Proposition \ref{p1} gives $$\beta_j\le 1-p_j^{-\mathsf v_{p_j}(n_1)}$$
 and  a factorization of
$AP$ with length
\[
 \ell^*:=\sum_{j=1}^u\max\mathsf L(A_j)
       =\sum_{j=1}^u\mathsf k(A_j)-(r-1)\sum_{j=1}^u\beta_j
       =\mathsf k(A)+C-(r-1)\sum_{j=1}^u\beta_j.
\]
Since $r-1\ge n_r-3$, Lemma \ref{l9} implies 
$\ell^*\le C+1$. Therefore
\begin{equation*}
 \mathsf k(A)\le1+(r-1)\sum_j\beta_j
              \le1+(r-1)\sum_{j=1}^u\left(1-p_j^{-\mathsf v_{p_j}(n_1)}\right)
              =1+(r-1)\sum_{p\mid n_1}\left(1-p^{-\mathsf v_p(n_1)}\right).
\end{equation*}
\end{proof}

We are now ready to prove Theorem \ref{t1} by combining Proposition \ref{p2}
with the lower bound in Lemma \ref{l5}.2.

\begin{proof}[Proof of Theorem \ref{t1}] Note that $r\ge 2$.
	If $(n_r,r)=(4,2)$, then $G$ is either $C_2\oplus C_4$
	or $C_4^2$. Since $\mathsf K(G,1)=\mathsf K(G)$, the assertion
	follows from Lemma \ref{l1}.2.
	If $n_r\le5$ and $(n_r,r)\neq (4,2)$, then $n_r$ is a prime power with $r\ge \lfloor n_r/2\rfloor+1$, whence
	the assertion follows from Lemma \ref{l5}.3.

	Suppose $n_r\ge6$. Since $r\ge n_r-2\ge\lfloor n_r/2\rfloor+1$,
	Lemma \ref{l5}.2 gives
	\[
	\mathsf K(G,r-1)\ge 1+(r-1)\sum_{p\mid n_1}
	\left(1-\frac1{p^{\mathsf v_p(n_1)}}\right).
	\]
	Let $G_0\subset G$  with $\Delta(G_0)\neq \emptyset$ be such that $r-1$ divides $\min\Delta(G_0)$ and let $A\in \mathcal A(G_0)$ with $\mathsf k(A)=\mathsf K(G_0)$.
	It suffices to show $$\mathsf k(A)\le 1+(r-1)\sum_{p\mid n_1}
	\left(1-\frac1{p^{\mathsf v_p(n_1)}}\right).$$ 
	We may assume that $\mathsf k(A)>1$. In view of $r\ge n_r-2\ge\lfloor n_r/2\rfloor+1$, it follows from Lemma \ref{l4}.2 that $\supp(A)$ is an LCN-set. By Lemma \ref{l2}.1 and 2, we have $\min\Delta(G_0)$ divides $\min\Delta(\supp(A))$ and $\min\Delta(\supp(A))\le \max\{r-1,n_r-2\}<2(r-1)$,  whence $\min\Delta(\supp(A))=r-1$. In view of $r\ge n_r-2\ge 4$, the assertion now follows from 
	Proposition \ref{p2}.
\end{proof}

\section{Subsets with minimum distance \texorpdfstring{$\exp(G)-3$}{exp(G)-3}}\label{s4}
In this section, we  consider subsets with
minimum distance $\exp(G)-3$. Every atom
of cross number less than one over such a subset has three distinct
terms. The following lemmas describe the subgroup they generate and
its intersection with the subgroup generated by the remaining support.

\begin{lemma}\label{l10}
Let $G=H\oplus K$ be a finite abelian group and let $G_0\subset G$ be a subset.
\begin{enumerate}
	\item If $G$ is a $p$-group and $H\cap \langle G_0\rangle \neq 0$, then there exists a subset
	$E\subset G_0$ with $|E|\le\mathsf r(K)+1$ such that
	$H\cap\langle E\rangle\ne0$.

	\item If $G^*$ is a subgroup of $G$ such that
	$G^*\cong C_{\exp(G)}^s$ for some $s\in \mathbb N$, then $G^*$
	is a direct summand.
\end{enumerate}

\end{lemma}
\begin{proof}
1. Let $\varphi_K\colon G\to K$ be the projection, and set
$M=\langle\varphi_K(G_0)\rangle$. A generating set of a finite
abelian $p$-group contains a generating subset of size at most its
rank. Thus there exists $E^*\subset G_0$ with
$|E^*|\le\mathsf r(M)\le\mathsf r(K)$ such that 
$\langle\varphi_K(E^*)\rangle=M$. In particular,
$G_0\subset H+\langle E^*\rangle$.
If $H\cap\langle E^*\rangle\ne0$, take $E=E^*$.
Otherwise $H+\langle E^*\rangle=H\oplus\langle E^*\rangle$.
Since $H\cap\langle G_0\rangle\ne0$, some
$g\in G_0$ lies outside $\langle E^*\rangle$. Write
$g=h+u$ with $h\in H$ and $u\in\langle E^*\rangle$.
Then 
\[
0\ne h\in H\cap\langle E^*\cup\{g\}\rangle.
\]
Thus $E=E^*\cup\{g\}$ has the required size.

2. The assertion follows by applying Szele's splitting criterion
\cite[Chapter 5, Lemma 2.1]{Fuchs15} to each primary component of $G$.
\end{proof}

\begin{lemma}\label{l11}
Let $G$ be a finite abelian group of exponent $n\ge6$, and let
$G_0\subset G$ be a subset with $\min\Delta(G_0)=n-3$. Suppose that $U\in\mathcal A(G_0)$ satisfies $\mathsf k(U)<1$.
\begin{enumerate}
\item We have $U=g_1 g_2 g_3$, where $g_1,g_2,g_3$ are distinct elements of order $n$ and $g_3=-g_1-g_2$.

\item We have $H:=\langle g_1,g_2,g_3\rangle\cong C_n^2$, $\mathcal A(\{g_1,g_2,g_3\})=\{g_1^n,g_2^n,g_3^n,U\}$, and $G_0\cap H\subseteq\{0,g_1,g_2,g_3\}$.

\item For $h\in G_0\setminus \{g_1,g_2,g_3\}$, we have $\langle h\rangle\cap H=0$.

	\item Let $U_1, U_2\in \mathcal A(G_0)$ be two distinct atoms with $\mathsf k(U_1), \mathsf k(U_2)<1$. Then $\supp(U_1)\cap \supp(U_2)=\emptyset$.
\end{enumerate}

\end{lemma}
\begin{proof}
Set $d=n-3$. Then $d\ge \lfloor n/2\rfloor\ge 3$ and Item 1  follows from
Lemma \ref{l4}.1.

2.  By Item 1, every proper subset of $\{g_1,g_2,g_3\}$
is an LCN-set and hence is half-factorial by Lemma \ref{l4}.3. Suppose
$V\in\mathcal A(\{g_1,g_2,g_3\})\setminus\{g_1^n,g_2^n,g_3^n,U\}$.
Since $U$ divides every sequence containing all three elements $\{g_1,g_2,g_3\}$, the support of $V$ is a proper subset of $\{g_1,g_2,g_3\}$, whence $\mathsf k(V)=1$.
Since $V\neq g_i^n$ for all $i\in [1,3]$, we may assume by symmetry that $V=g_1^xg_2^{n-x}$, where $x\in [1,\lfloor n/2\rfloor]$.
Then 
\[
V\cdot g_3^n=U^x\cdot(g_2^{n-2x}g_3^{n-x}) \quad \text{ and }\quad  \mathsf L(g_2^{n-2x}g_3^{n-x})=\{\mathsf k(g_2^{n-2x}g_3^{n-x})\}=\left\{\frac{2n-3x}{n}\right\}\,,
\]
whence $\frac{2n-3x}{n}=1$, $x=\frac{n}{3}\ge 2$, and $d$ divides $(x+1)-2=\frac{n-3}{3}<n-3$, a contradiction.
Therefore
$\mathcal A(\{g_1,g_2,g_3\})=\{g_1^n,g_2^n,g_3^n,U\}$.
In particular, $g_1,g_2$ are independent and $H\cong C_n^2$.

Suppose $h=ag_1+bg_2\in(G_0\cap H)\setminus\{0,g_1,g_2,g_3\}$, with $0\le a,b\le n-1$. Each of $\{g_1,g_2,h\}$, $\{g_2,g_3,h\}$,
and $\{g_3,g_1,h\}$ is an LCN-set, since an atom with cross number less than one
would be the product of its three distinct elements, forcing $h\in \{g_1,g_2,g_3\}$.
Each set is therefore half-factorial. The sequences
$h g_1^{n-a} g_2^{n-b}$, $h g_2^{n+a-b} g_3^a$, and
$h g_1^{n+b-a} g_3^b$ are zero-sum, so their cross numbers are integers
\[
\frac{1}{\ord(h)}+\frac{2n-a-b}{n},\qquad
\frac{1}{\ord(h)}+\frac{n+2a-b}{n},\qquad
\frac{1}{\ord(h)}+\frac{n+2b-a}{n}\in\mathbb Z,
\]
whence
\[
 a+b\equiv a-2b\equiv b-2a\equiv n/\ord(h)\pmod n.
\]
Therefore $3a\equiv3b\equiv0\pmod n$, so $3h=0$ and $\ord(h)=3$.
Now $a,b\in\{0,n/3,2n/3\}$, and hence $a+b\equiv n/3 \pmod n$ implies that either $\{a,b\}=\{0, n/3\}$ or $a=b=2n/3$. Thus $h=(n/3)g^*$ for some $g^*\in \{g_1,g_2,g_3\}$.
By symmetry, we may assume that
$h=(n/3)g_1$.
 The sequences
$B:=h {g_2}^{n/3} {g_3}^{n/3}$ and $C:=h^2 g_1^{n/3}$ are atoms with cross number one. Thus
$B\cdot C=h^3\cdot U^{n/3}$ implies that $d$ divides $(n/3+1)-2=(n-3)/3<n-3$,
a contradiction. Thus $G_0\cap H\subseteq\{0,g_1,g_2,g_3\}$.
\medskip

3. Let $h\in G_0\setminus \{g_1,g_2,g_3\}$. If $h=0$, the assertion is trivial. Thus we may suppose $h\neq 0$.  By Item 2, we have $h\notin H$. Assume to the contrary that $\langle h\rangle\cap H\ne0$. Let $a$ be the
least positive integer such that $h':=ah\in H$. Then $h'\ne0$ and
$\ord(h')=\ord(h)/a<n$, so $h'\notin \{g_1,g_2,g_3\}$. By the minimality of $a$,
every zero-sum sequence over $\{h,g_1,g_2,g_3\}$ contains $h$ with
multiplicity divisible by $a$. Replacing each $h^a$ by $h'$ gives
an isomorphism between the zero-sum monoids over $\{h,g_1,g_2,g_3\}$
and $\{h',g_1,g_2,g_3\}$ that preserves factorization lengths, whence
\[
\min\Delta(\{h,g_1,g_2,g_3\})=\min\Delta(\{h',g_1,g_2,g_3\})\,.
\]
Since $\{h,g_1,g_2,g_3\}\subset G_0$ and
$\min\Delta(\{g_1,g_2,g_3\})=\min\Delta(G_0)=d$, we obtain
\[
\min\Delta(\{h',g_1,g_2,g_3\})=\min\Delta(\{h,g_1,g_2,g_3\})=d.
\]
By applying item 2 to $\{h',g_1,g_2,g_3\}$ (instead of $G_0$), we obtain
$\{h',g_1,g_2,g_3\}\cap H\subseteq \{0,g_1,g_2,g_3\}$, a contradiction.

4. Let $U_1,U_2\in\mathcal A(G_0)$ be distinct atoms of cross number
less than one. By Item 1, write $U_1=g_1 g_2 g_3$ and
$U_2=h_1 h_2 h_3$, where all terms have order $n$.
If their supports meet in two elements, the zero-sum property makes
their third elements equal, contrary to $U_1\ne U_2$.
Suppose they meet in one element, say $g_3=h_3$.
Item 3 gives $h_1,h_2\notin \langle g_1,g_2,g_3\rangle $, whence
$g_1^{n-1} g_2^{n-1} h_1h_2$ is an atom with cross number two.
By Lemma \ref{l4}.2, $\{g_1,g_2,h_1,h_2\}$ is an LCN-set and
\[
\min\Delta(\{g_1,g_2,h_1,h_2\})\le 4-2=2<d,
\]
a contradiction.
\end{proof}

\begin{lemma}
\label{l12}
Let $G$ be a finite abelian group of exponent $n\ge6$, and let
$G_0\subset G$ satisfy $\min\Delta(G_0)=n-3$. Let
$U\in\mathcal A(G_0)$ have $\mathsf k(U)<1$. Then every
$E\subset G_0\setminus\supp(U)$ with $|E|\le n-3$ satisfies
\[
 \langle\supp(U)\rangle\cap\langle E\rangle=0.
\]
\end{lemma}
\begin{proof}
By Lemma \ref{l11}.1, we let $U=g_1 g_2 g_3$, where
$g_1,g_2,g_3$ are distinct elements of order $n$. Set 
\[
\text{$H=\langle g_1,g_2,g_3\rangle$\quad  and \quad $d=n-3$.}
\]
Then $d=\min\Delta(G_0)\ge \lfloor n/2\rfloor\ge 3$ and $H\cong C_n^2$ by Lemma \ref{l11}.2.
Assume to the contrary that there is a subset
$E\subset G_0\setminus \{g_1,g_2,g_3\}$ with $|E|\le d$ such that
$H\cap\langle E\rangle\ne0$. Choose $E$ inclusion-minimal with
$J:=H\cap\langle E\rangle\ne0$. Then 
\begin{equation}\label{le-e-1}
g\not\in \langle E\setminus\{g\}\rangle\quad  \text{ for every }\quad g\in E.
\end{equation}
Let $A\in \mathcal A(E\cup \{g_1,g_2,g_3\})$ with $A\neq U$ and $\supp(A)\subsetneq E\cup \{g_1,g_2,g_3\}$.
Suppose $\mathsf k(A)<1$. By Lemma \ref{l11}.1 again, we have $A=h_1h_2h_3$ with $h_3=-h_1-h_2$.
Then \eqref{le-e-1} implies that $\supp(A)\not\subset E$ and Lemma \ref{l11}.4 implies that $\supp(A)\cap \{g_1,g_2,g_3\}=\emptyset$, a contradiction. Thus $\mathsf k(A)\ge 1$ and hence
\begin{equation}\label{le-e-2}
	E\cup \{g_1,g_2,g_3\}\setminus\{g_i\}\quad  \text{ is an LCN-set for every }\quad i\in [1,3].
\end{equation}

If $|E|<d$, then Lemma \ref{l4}.3 implies that those
three LCN-sets are half-factorial.   Choose $T\in\mathcal F(E)$ with $0\ne\sigma(T)\in J$. Since $H\cong C_n^2$ and $g_3=-g_1-g_2$, there exist distinct $i,j\in [1,3]$ and $a,b\in [1,n-1]$ such that $-\sigma(T)=ag_i+bg_j$.
After renumbering  $g_1,g_2,g_3$ if necessary, we may assume that 
$-\sigma(T)=ag_1+bg_2$ with $1\le a\le  b\le n-1$. 
Since $A:=T g_1^a g_2^b$ is zero-sum and both $\supp(A), \supp(Ag_3^nU^{-a})$ are half-factorial by \eqref{le-e-2},  it follows from
$A\cdot g_3^n=U^a\cdot (Ag_3^nU^{-a})$ that $d$ divides $(a+\mathsf k(Ag_3^nU^{-a}))-(\mathsf k(A)+1)=\frac{a(n-3)}{n}$, whence $a=0$, a contradiction. 
Therefore 
\begin{equation}\label{le-e-3}
	|E|=d\,.
\end{equation}
By Lemma \ref{l4}.3, we have
\begin{equation}\label{le-e-11}
E\cup \{g_i\} \text{ is half-factorial for every }i\in [1,3].
\end{equation}


For each $e\in E$, consider
\[
\varphi_e:J\longrightarrow
\langle E\rangle/\langle E\setminus\{e\}\rangle,
\qquad \text{ defined by }\quad  j\longmapsto j+\langle E\setminus\{e\}\rangle.
\]
Its kernel is $J\cap \langle E\setminus\{e\}\rangle=H\cap\langle E\setminus\{e\}\rangle=0$ by the choice of $E$. Thus $J$ embeds in a cyclic group of order $$\ord\big(e+\langle E\setminus\{e\}\rangle\big)\le \ord(e)\le n.$$ In particular, 
\begin{equation}\label{le-e-7}
\text{$\varphi_e$ is injective\quad  and \quad  $J$ is cyclic with $|J|\le \ord\big(e+\langle E\setminus\{e\}\rangle\big)\le  \ord(e)\le  n$.}
\end{equation}

We proceed by proving four claims.

\begin{claim}\label{c4}
	If $|J|=n$, then $E$ is independent and every element has order $n$.
\end{claim}

\begin{proof}[Proof of \ref{c4}]
	Suppose $|J|=n$. Then \eqref{le-e-7} implies that $\ord\big(e+\langle E\setminus\{e\}\rangle\big)=\ord(e)=n$ 
	for every $e\in E$. If  $\sum_{e\in E}a_ee=0$ for some $a_e\in [0, \ord(e)-1]$, then for each $e_0\in E$, we have
	$$0+\langle E\setminus\{e_0\}\rangle=\sum_{e\in E}(a_ee+\langle E\setminus\{e_0\}\rangle)=a_{e_0}e_0+\langle E\setminus\{e_0\}\rangle=a_{e_0}\left(e_0+\langle E\setminus\{e_0\}\rangle\right),$$ whence $\ord\big(e_0+\langle E\setminus\{e_0\}\rangle\big)=n$  divides $a_{e_0}$ and hence
	\[
	\text{$a_{e_0}=0$ \quad for every $e_0\in E$,}
	\]  implying 
	 $E$ is independent.
	\qedhere[End of \ref{c4}]
\end{proof}

\begin{claim}\label{c1}
$J\cap\langle v\rangle=0$ for every $v\in \{g_1,g_2,g_3\}$.
\end{claim}

\begin{proof}[Proof of \ref{c1}]
	Let $v\in \{g_1,g_2,g_3\}$. Assume to the contrary that  $h:=|J\cap\langle v\rangle|=|\langle E\rangle\cap \langle v\rangle|>1$. Then  $a:=n/h$ is the least positive integer such that $av\in \langle E\rangle$. Choose a zero-sum free sequence $T\in \mathcal F(E)$ such that  $A=v^a\cdot T$ is zero-sum. Then $A$ is an atom with $\supp(A)=E\cup\{v\}$ by the minimality of $E$. 
	It follows from \eqref{le-e-11} that $\supp(A)$ is half-factorial, whence $$\mathsf k(T)=1-1/h=(n-a)/n.$$
	
	If $\mathsf v_g(T)\ge 2$ for all $g\in E$, then $n-a=n\mathsf k(T)\ge |T|\ge 2|E|=2(n-3)\ge n$, a contradiction. Thus 
	\begin{equation}\label{le-e-5}
	\text{	$\mathsf v_{g^*}(T)=1$ for some $g^*\in E$.}
	\end{equation}
Suppose  $\{g_1,g_2,g_3\}=\{v,w,t\}$. Then $w+t=-v$ and
 the sequence
$B:=w^{n-a} t^{n-a} T$ is zero-sum. Assume that $B=B_1B_2$, where $B_1,B_2$ are nontrivial zero-sum subsequences. Then \eqref{le-e-5} implies that there exists $i\in [1,2]$, say $i=1$, such that $\mathsf v_{g^*}(B_1)=0$. Since $w^{n-a} t^{n-a}$ is zero-sum free, we obtain $B_1=W_1W_2$ with $W_1\mid w^{n-a} t^{n-a}$ and $1\neq W_2\mid T$. Then $0\neq \sigma(W_2)=-\sigma(W_1)\in H\cap \langle E\setminus\{g^*\}\rangle$, a contradiction to the minimality  of $E$. Thus 
\[
\text{ $B$\quad  is an atom with\quad  $\mathsf k(B)=\frac{2n-2a}{n}+\mathsf k(T)=\frac{3n-3a}{n}$.}
\]
Set $B'=w^a t^a T^{h-1}$. Since its support is an LCN-set by \eqref{le-e-2}, every $\ell\in\mathsf L(B')$ satisfies
\[
 1\le\ell\le \mathsf k(B') =\frac{2a}{n}+(h-1)\mathsf k(T)= h-2+\frac3h<h.
\]
Since $B\cdot B'=w^n\cdot t^n\cdot T^h$ and $E$ is half-factorial, we obtain that $d$ divides $(h+1)-(1+\ell)=h-\ell$.
 If $h<n$, then $0<h-\ell\le n/2-1<d$,
a contradiction. Thus  $h=n$ and $a=1$. Since $\mathsf L(B^n)\supset \{n, n\mathsf k(B)\}$, we obtain that $d=n-3$ divides $n(\mathsf k(B)-1)=2n-3$, forcing 
\[
\text{$n=6$\quad  and\quad  $|E|=d=3$.}
\] In view of $h=n=6$ and \eqref{le-e-7}, it follows from $\langle v\rangle\subseteq J$ that $|\langle v\rangle| = |J|=n=6$. By \ref{c4}, 
\begin{equation}\label{le-e-8}
	\text{the three elements of $E$ are independent and
		have order six. }
\end{equation}

Note that $\sigma(T)=-av=-v$. Let $e\in E$. Then \eqref{le-e-7} implies that $\varphi_e$ is injective  and hence $\ord(\varphi_e(-v))=\ord(v)=n$. 
Since $$\varphi_e(-v)=\varphi_e(\sigma(T))=\mathsf v_e(T)e+\langle E\setminus\{e\}\rangle\,,$$
it follows from \eqref{le-e-8} that 
\[
n=\ord(\varphi_e(-v))=\ord\big(\mathsf v_e(T)e+\langle E\setminus\{e\}\rangle\big)=\ord\big(\mathsf v_e(T)e\big)=\frac{n}{\gcd(\mathsf v_e(T), n)},
\]
whence $\gcd(\mathsf v_e(T), n)=1$. Since $n=6$, we obtain 
\begin{equation}\label{le-e-9}
\text{$\mathsf v_e(T)\in\{1,5\}$ for every $e\in E$.}
\end{equation}
In view of  \eqref{le-e-8} and  $\mathsf k(T)=1-1/6=5/6$, it follows that 
$\sum_{e\in E}\mathsf v_e(T)=|T|=5$, a contradiction to \eqref{le-e-9}.
\qedhere[End of \ref{c1}]
\end{proof}

\begin{claim}\label{c2}
Let $T\in\mathcal F(E)$ with $0\ne\sigma(T)\in J$ satisfy $\mathsf v_e(T)<\ord(e)$ for
every $e\in E$,  and let $v,w\in \{g_1,g_2,g_3\}$ be distinct such that 
$T v^a w^b$ is zero-sum for some $1\le a,b\le n-1$.  Then $T v^a w^b$ is an atom.
\end{claim}

\begin{proof}[Proof of \ref{c2}]
We may assume that $v=g_1$ and $w=g_2$.	Suppose $V$ is an atom with $\mathsf k(V)=1$ and $\supp(V)=E\cup\{g_1,g_2\}$.
 Since $|E|=d=n-3$, we obtain 
$\mathsf v_{g_1}(V)+\mathsf v_{g_2}(V)\le3$.
If $\mathsf v_{g_1}(V)=\mathsf v_{g_2}(V)=1$, set $W:=Vg_3^nU^{-1}$. Then
 $\supp(W)=E\cup \{g_3\}$ is half-factorial by \eqref{le-e-11}, a contradiction to the fact that 
$\mathsf k(W)=2-\frac{3}{n}$ is not an integer. Otherwise, after renumbering if necessary, we may assume that 
\[
V=g_1^2 g_2\prod_{e\in E}e,
\qquad  \text{ with }\quad -2g_1-g_2=\sum_{e\in E}e\in J.
\]
The cross number $\mathsf k(V)=1$ implies that every element in the support of $V$ has order $n$.
Since $-2g_1-g_2$ has order $n$, we have
$J=\langle-2g_1-g_2\rangle$ and 
\[
\text{$E$ is independent by \ref{c4}.}
\]
If $n$ is even, then
$0\neq (n/2)(-2g_1-g_2)\in J\cap\langle g_2\rangle$, a contradiction to \ref{c1}. Thus $n$ is odd.
Set
\[
W=g_1 g_2^{(n+1)/2}\prod_{e\in E}e^{(n+1)/2}.
\]
If $W$ is not an atom, then we can choose  a proper
zero-sum subsequence $W_1$ with $g_1\not\in \supp(W_1)$. Since $E$ is independent, we infer that $g_2\in \supp(W_1)$, whence $0\neq xg_2\in J\cap\langle g_2\rangle$
for some $x\in[1,(n+1)/2]$, contrary to \ref{c1}. Thus $W$ is an atom. Since 
\[
V^{(n+1)/2}=g_1^n\cdot W,
\]
we obtain
\[
 d\mid \frac{n+1}{2}-2=\frac{n-3}{2}=\frac d2,
\]
a contradiction. Thus we have shown
\begin{equation}\label{le-e-12}
\text{ there is no atom $V$ with $\mathsf k(V)=1$ and $\supp(V)=E\cup\{g_1,g_2\}$.}
\end{equation}

Let $A:=Tv^aw^b=T g_1^a g_2^b$.  Choose $s\in \mathsf L(A)$ and $t\in \mathsf L(A^{-1}\prod_{u\in E\cup\{g_1,g_2\}}u^{\ord(u)})$.
Then
\[
 2\le s+t\le |E\cup\{g_1,g_2\}|= d+2, \qquad  \text{ with }\quad s+t\equiv d+2\pmod d.
\]
If $s+t=d+2$, then $A$ is a product of $s$ atoms having cross number one, whence we may choose an atom $A_0\mid A$ with $\mathsf k(A_0)=1$ and $\mathsf v_{g_1}(A_0)+\mathsf v_{g_2}(A_0)\neq 0$.  The minimality of $E$ implies that $E\subset \supp(A_0)$. In view of \eqref{le-e-12}, we obtain $\supp(A_0)\neq E\cup\{g_1,g_2\}$, whence, after renumbering if necessary, we may assume that 
 $\supp(A_0)=E\cup\{g_1\}$. Therefore
$0\neq xg_1\in J\cap\langle g_1\rangle$ for some $x\in[1,n-1]$,
contrary to \ref{c1}. Thus $s=t=1$ and hence $A=Tv^aw^b$ is an atom.
\qedhere[End of \ref{c2}]
\end{proof}

\begin{claim}\label{c3}
$J$ has no element of order $m$ with $1<m<n$.
\end{claim}

\begin{proof}[Proof of \ref{c3}]
Assume to the contrary that there exists $T\in\mathcal F(E)$ with
$\mathsf v_e(T)<\ord(e)$ for all
$e\in E$ such that $\sigma(T)\in J$ has order $m$ with $1<m<n$.
Set $m'=n/m$. Since $H=\langle g_1,g_2\rangle\cong C_n^2$ by Lemma \ref{l11}.2,
we may write $-\sigma(T)=am'g_1+bm'g_2$ with $0\le a,b<m$.
By \ref{c1} and $g_3=-g_1-g_2$, we obtain $a\neq 0$, $b\neq 0$, and $a\neq b$.
After renumbering if necessary, we may assume $$1\le a<b\le m-1.$$
By \ref{c2}, the zero-sum sequences
\[
A_1:=T g_1^{m'a} g_2^{m'b},\quad
A_2:=T g_2^{m'(b-a)} g_3^{n-m'a},\quad
A_3:=T g_1^{n+m'(a-b)} g_3^{n-m'b}
\]
are all atoms.

Since  $J\cap\langle g_2\rangle=0$ by \ref{c1}, we obtain that $-\sigma(T)+\langle g_2\rangle \in H/\langle g_2\rangle$ has order 
$\ord(-\sigma(T))=m$. Note that $-\sigma(T)+\langle g_2\rangle=am'g_1+\langle g_2\rangle$ has order $\ord(am'g_1)=m/\gcd(a,m)$. Thus 
$\gcd(a,m)=1$.
Since $E$ is half-factorial, we have $\mathsf L(T^m)=\{m\mathsf k(T)\}$ and hence 
\[
\mathsf L(A_1^m)\supset \{m, a+b+m\mathsf k(T)\},\quad \mathsf L(A_2^m)\supset \{m, m+b-2a+m\mathsf k(T)\},\quad
\mathsf L(A_3^m)\supset \{m, 2m+a-2b+m\mathsf k(T)\}.
\]
 Subtracting the resulting length differences
gives
\begin{equation}\label{le-e-21}
 d\mid m-3a,\qquad \text{ and }\quad  d\mid2m-3b.
\end{equation}
Since $m$ is a proper divisor of $n$, we have $m\le n/2$, whence
$d=n-3\ge2m-3$. Together with $1\le a\le m-2$ and
$2\le b\le m-1$, this gives
\[
-d<m-3a<d,\qquad \text{ and }\quad -d<2m-3b<d.
\]
By \eqref{le-e-21}, we obtain $m=3a$ and $2m=3b$.
Thus $\gcd(a,m)=1$ implies  
\[
\text{$a=1$, $m=3$, and $b=2$,}
\]
whence $m'=n/3$ and $-\sigma(T)=m'g_1+2m'g_2$.
Let
\[
c(T)=\{e\in E\colon 2\mathsf v_e(T)\ge\ord(e)\}.
\]
Then there exists $T_2\in \mathcal F(E)$ such that 
\[
T^2=\left(\prod_{e\in c(T)}
e^{\ord(e)}\right)T_2, 
\qquad \text{ with } \mathsf v_e(T_2)<\ord(e)\quad \text{ for every }e\in E.
\]
 Since $\sigma(T_2)=2\sigma(T)=-2m'g_1-m'g_2\ne0$, it follows from \ref{c2} that  $T_2 g_1^{2m'} g_2^{m'}$ and $T_2g_2^{2m'}g_3^{m'}$ are atoms. Consequently
\[
U^{m'}\cdot  (T_2 g_1^{2m'} g_2^{m'})=g_1^n\cdot (T_2 g_2^{2m'} g_3^{m'})\,,
\]
whence $d=n-3$ divides $(m'+1)-2=\frac{n-3}{3}$, a contradiction.
\qedhere[End of \ref{c3}]
\end{proof}

By \eqref{le-e-7}, we have $1<|J|\le n$, whence \ref{c3} implies that $|J|=n$ is prime. In particular, $n\ge7$.
Let $0\ne h\in J$. Choose $T\in\mathcal F(E)$ with
$\mathsf v_e(T)<\ord(e)$ for every $e\in E$ such that $-\sigma(T)=h$.
By symmetry, $H=\langle g_1,g_2\rangle$, and \ref{c1}, we may assume $h=ag_1+bg_2$ with
$1\le a<b\le n-1$.  Then \ref{c2} implies that
\[
A_1:=T g_1^a g_2^b,\qquad
A_2:=T g_2^{b-a} g_3^{n-a},\qquad
A_3:=T g_1^{n+a-b} g_3^{n-b}
\]
are all atoms. Since $E$ is half-factorial, we obtain 
\[
\mathsf L(A_1^n)\supset\{n,a+b+n\mathsf k(T)\},\quad \mathsf L(A_2^n)\supset\{n,n+b-2a+n\mathsf k(T)\},\quad
\mathsf L(A_3^n)\supset\{n,2n+a-2b+n\mathsf k(T)\}\,,
\]
whence 
\[
n\equiv a+b+n\mathsf k(T)\equiv n+b-2a+n\mathsf k(T)\equiv 
2n+a-2b+n\mathsf k(T) \pmod d\,.
\]
It follows that  $d\mid n-3a$ and $d\mid2n-3b$.
Since $n\equiv3\pmod d$ and $n\ge 7$ is prime, we conclude $\gcd(3,d)=1$ and hence
\begin{equation}\label{le-e-31}
a\equiv1\pmod d,\qquad \text{ and }\quad  b\equiv2\pmod d.
\end{equation}
This holds for every nonzero $h\in J$ after choosing $a<b$.
Now take $h=ag_1+bg_2\ne0$. Since $n$ is prime, there exists $k$ such that $ka\equiv3\pmod n$, and hence
$kh=3g_1+b'g_2$ for some $b'\in[1,n-1]$.
Applying \eqref{le-e-31} to $kh$, we obtain either
$3\equiv1 \pmod d$ or $3\equiv2\pmod d$, a contradiction to  $d=n-3\ge4$.
\end{proof}

\section{Elasticity with respect to minimal distances}\label{s5}
In this section, we compare $\rho^*(G,d)$ and $\mathsf K(G,d)$ for
values of $d$ close to the rank or exponent of $G$.
We first apply the results of Section \ref{s4} to subsets with
minimum distance $\exp(G)-3$, and then study subsets with minimum
distance $\exp(G)-2$. These comparisons will be used in
Section \ref{s6} to prove Theorem \ref{t2}.

\begin{proposition}\label{p3}
Let $G$ be a finite abelian group with $\exp(G)=n\ge 6$ and let
$d=n-3$. If either $\mathsf r(G)=n-2$, or
$G\cong C_2\oplus C_{n_2}\oplus\cdots\oplus C_{n_{n-1}}$ with
$2\mid n_2\mid\cdots\mid n_{n-1}=n$, then
\[
d\in \Delta_1(G) \quad \text{ and }\quad  \rho^*(G,d)=\mathsf K(G,d).
\]
\end{proposition}
\begin{proof}In view of Lemma \ref{l2}, we obtain that $d=n-3\in [1, \mathsf r(G)-1]\subset \Delta^*(G)\subset \Delta_1(G)$. By Lemma \ref{l6}.1, we have $\rho^*(G,d)\ge \mathsf K(G,d)$, whence it suffices to show $\rho(G_0)\le \mathsf K(G, d)$ for every  $G_0\subset G$ with $\Delta(G_0)\ne\emptyset$ and
$d\mid\min\Delta(G_0)$. By Lemma \ref{l6}.2 and 4, we may assume that 
\[
\text{$G_0$ is not an LCN-set \quad and \quad $0\not\in G_0$.}
\]
Then $\min\Delta(G_0)=d$, because
$\max\Delta^*(G)=n-2<2d$. Choose
$U\in\mathcal A(G_0)$ with $\mathsf k(U)<1$.
By Lemma \ref{l11}, we have $U=g_1 g_2 g_3$ and
$H=\langle g_1,g_2,g_3\rangle\cong C_n^2$.
By Lemma \ref{l10}.2, there exists a subgroup $K$ of $G$ such that $G=H\oplus K$.
Set 
\begin{align*}
&\text{$R=G_0\setminus \{g_1,g_2,g_3\}$,\quad 
	$J=H\cap\langle R\rangle$,}\\
&\text{$J_p, K_p, H_p, G_p$ be the $p$-primary components of $J, K, H, G$ respectively,}\\
	\text{ and\quad }&\text{$\pi_p\colon G\rightarrow G_p$ be the projection for every prime $p\mid n$.}
\end{align*}

In the first case that $\mathsf r(G)=n-2$, we have $\mathsf r(K_p)=\mathsf r(G_p)-2\le n-4$ for every
prime $p\mid n$. If $J\ne0$, one of its primary components is
nonzero, say $J_p=H_p\cap \langle \pi_p(R)\rangle\neq 0$. Apply Lemma \ref{l10}.1 to
$G_p=H_p\oplus K_p$ and $\pi_p(R)$. There exists
$E\subset R$ with $|E|\le n-3$ such that 
$H_p\cap\langle\pi_p(E)\rangle\ne0$, whence $H\cap\langle E\rangle\ne0$, contrary to
Lemma \ref{l12}. Thus $J=0$.

In the second case that $G\cong C_2\oplus C_{n_2}\oplus\cdots\oplus C_{n_{n-1}}$ with
$2\mid n_2\mid\cdots\mid n_{n-1}=n$,  for every odd prime $p\mid n$,  we have $\mathsf r(K_p)\le \mathsf r(G_p)-2\le \mathsf r(C_{n_2}\oplus\cdots\oplus C_{n_{n-1}})-2\le n-4$. The same argument as above shows that 
\[
\text{$J_p=0$\quad  for every odd prime $p\mid n$.}
\]
If $2\cdot J\ne0$, then $\mathsf v_2(n)\ge 2$,
\[
\text{$0\neq 2\cdot J_2\subset 2\cdot H_2\cap \langle 2\cdot \pi_2(R)\rangle$,\quad  and \quad 
	$\mathsf r(2\cdot K_2)\le \mathsf r(2\cdot G_2)-2\le \mathsf r (C_{n_2}\oplus\cdots\oplus C_{n_{n-1}})-2= n-4$.}
\]  
Apply Lemma \ref{l10}.1 to
$2\cdot H_2\oplus2\cdot K_2$ and $2\cdot \pi_2(R)$. There exists a subset $E\subset R$ with $|E|\le n-3$ such that 
$2\cdot H_2\cap \langle 2\cdot \pi_2(E)\rangle\neq 0$, whence 
$H\cap\langle E\rangle\ne0$, again contradicting
Lemma \ref{l12}. Thus  $2\cdot J=0$.

If $J\ne0$, choose $E\subset R$ inclusion-minimal with
$J_E=H\cap\langle E\rangle\ne0$.  Choose $e\in E$ and consider
\[
\varphi:J_E\longrightarrow
\langle E\rangle/\langle E\setminus\{e\}\rangle,
\qquad \text{ defined by }\quad  j\longmapsto j+\langle E\setminus\{e\}\rangle.
\]
Its kernel is $J_E\cap \langle E\setminus\{e\}\rangle=H\cap\langle E\setminus\{e\}\rangle=0$ by the choice of $E$. Thus $J_E$ embeds in a cyclic group $\langle E\rangle/\langle E\setminus\{e\}\rangle$, whence
 $J_E$ is cyclic.
Since $2\cdot J=0$, we obtain $|J_E|=2$. Its nonzero element is
$(n/2)v$ for some $v\in \{g_1,g_2,g_3\}$, because
$(n/2)g_1,(n/2)g_2,(n/2)g_3$ are all elements of order $2$ in $H$.
After renumbering if necessary, we may assume that $J_E=\{0, (n/2)g_1\}$. 
Choose $T\in\mathcal F(E)$ with minimal length such that
$g_1^{n/2} T$ and hence $g_2^{n/2}g_3^{n/2}T$ are zero-sum.
By the minimality of $E$ and $|T|$, both sequences are atoms.
Then
\[
(g_1^{n/2} T)\cdot {g_2}^n\cdot {g_3}^n=U^{n/2}\cdot (g_2^{n/2} g_3^{n/2} T)
\]
implies that $d=n-3$ divides $n/2-2<n-3$, a contradiction.
Therefore
\[
\text{$J=H\cap \langle R\rangle  =\langle g_1,g_2,g_3\rangle\cap \langle G_0\setminus \{g_1,g_2,g_3\} \rangle=0$\quad  in both cases.
}
\] 
It follows that  $\mathcal A(G_0)=\mathcal A(\{g_1,g_2,g_3\})\cup \mathcal A(R)$ and that $\langle R\rangle$ is isomorphic to a subgroup of  $K$, whose rank is at most $n-3$. By Lemma \ref{l2}.2,
\[
\mathsf m(\langle R\rangle)\le  \mathsf m(K)\le\max\{n-4,\lfloor n/2\rfloor-1\}<n-3=d.
\]
If there exists $A\in \mathcal A(R)$ with $\mathsf k(A)>1$, 
then by Lemma \ref{l11}.1, any atom $A^*\in \mathcal A(\supp(A))$ with $\mathsf k(A^*)<1$ would divide $A$, whence $\supp(A)$ is an LCN-set and  
also non-half-factorial with minimum distance
at least $d$, contradicting the bound on $\mathsf m(\langle R\rangle)$.
Thus every atom over $G_0$ has cross number at most one, whence Lemma \ref{l11}.1 implies that 
\[
\{\mathsf k(A)\colon A\in \mathcal A(G_0)\}=\{1, 3/n\}\,.
\]
 Comparing cross
numbers gives $\rho(G_0)\le n/3$, while Lemma \ref{l6}.1 implies that 
$\mathsf K(G,d)\ge 1+d/2=1+(n-3)/2>n/3$. The assertion now follows.
\end{proof}

\begin{proposition}\label{p4}
Let $G\cong C_{n_1}\oplus\cdots\oplus C_{n_r}$, where
$1<n_1\mid\cdots\mid n_r$ and $r\ge\max\{2,n_r-2\}$.
Then
\begin{equation}\label{e13}
 \rho^*(G,r-1)=\mathsf K(G,r-1)
\end{equation}
in each of the following cases:
\begin{enumerate}
 \item $r\ge n_r$;
 \item $r=n_r-1\ge3$, unless $n_1=n_r$ is a prime power;
 \item $r=n_r-2\ge 4$.
\end{enumerate}
In the exceptional case of item 2, namely $G\cong C_n^{n-1}$ with
$n\ge4$ a prime power, we have
\[
 \rho^*(G,n-2)=1+\frac{n(n-2)}{n+1}
       >\mathsf K(G,n-2)=1+\frac{(n-1)(n-2)}n.
\]
\end{proposition}

\begin{proof}
By Lemma \ref{l2}.3, we have  $r-1\in \Delta^*(G)\subset \Delta_1(G)$.

If $r\ge n_r$, then $r-1>n_r-2$  and  the assertion follows from 
Lemma \ref{l6}.3.

If $r=n_r-1\ge3$, then $r-1\ge\max\{\lfloor n_r/2\rfloor,\lfloor(r+1)/2\rfloor\}$ and 
Lemma \ref{l7}.2 gives both equality and its stated
homocyclic prime-power exception. 

If $r=n_r-2\ge 4$, the assertion follows from
Proposition \ref{p3}.
\end{proof}

\medskip\noindent\textit{Remark.}
For completeness, $\min\Delta(C_2\oplus C_4)=1$ and
$\mathsf D(C_2\oplus C_4)=5$, so
$\rho^*(C_2\oplus C_4,1)=\rho(C_2\oplus C_4)=5/2$, whereas
$\mathsf K(C_2\oplus C_4,1)=3/2$ by Theorem \ref{t1}.
\medskip

\begin{lemma}\label{l13}
Let $G$ be a finite abelian group of exponent $n\ge6$,  let
$G_0\subset G$ with $\min\Delta(G_0)=n-2$ satisfy $\{g,-g\}\subset G_0$ for some $g$ with $\ord(g)=n$, and let
$E\subset G_0\setminus\{g,-g\}$ with $|E|\le n-1$ satisfy
$\langle g\rangle\cap\langle E\rangle\ne0$.
Then $|E|=n-1$, the elements of $E$ are independent of
order $n$, and
$\sum_{e\in E}e\in \{g,-g\}$.
\end{lemma}
\begin{proof}
We may assume that $E$ is inclusion-minimal, since the conclusion gives
$n-1$ elements for such a subset, so it must equal $E$.
Set $d=n-2$.
Let $a$ be the least positive integer such that $ag\in\langle E\rangle$,
and let $T\in \mathcal F(E)$ be zero-sum free such that $V:=g^aT$ is zero-sum. Since $\langle g\rangle\cap\langle E\rangle\ne0$, we obtain $a\in [1, n-1]$. By the minimality of $E$ and $a$, we have $\supp(T)=E$ and $V$ is an atom.

 Set $B=(-g)^{n-a} T$.
Every atom in a factorization of $B$ contains $-g$, so  $1\le\ell\le n-a$ for every $\ell\in \mathsf L(B)$. Thus 
$V\cdot (-g)^{n}=(g(-g))^a\cdot B$ implies  $d\mid a+\ell-2$ for every $\ell\in \mathsf L(B)$.
Consequently, 
\[
\text{either $a=\ell=1$\quad  or\quad  $\ell=n-a$.}
\]

If $\ell=n-a$, then we may suppose $B=B_1\cdot\ldots\cdot B_{n-a}$ with $B_i\in \mathcal A(G_0)$ for all $i\in [1, n-a]$. Since $T$ is zero-sum free, we obtain $\mathsf v_{-g}(B_i)=1$ for all $i\in [1, n-a]$, whence
$g\in\langle E\rangle$ and hence $a=1$. By the minimality of $E$, we have $E\subset \supp(B_i)$ for all $i\in [1,n-1]$ and hence $\mathsf v_e(T)\ge n-1$ for every $e\in E$. It follows that 
\[
B_i=(-g)\prod_{e\in E}e, \quad T=\prod_{e\in E}e^{n-1}, \text{ and }  \ord(e)=n\quad  \text{ for every } e\in E \text{ and  for every }i\in [1,n-1],
\]
whence $\sum_{e\in E}e=g$ and $E$ is independent since $T=\prod_{e\in E}e^{n-1}$ is zero-sum free.
If $\mathsf k((-g)\prod_{e\in E}e)=\frac{1+|E|}{n}<1$, then
Lemma \ref{l4}.1 would imply  $(-g)\prod_{e\in E}e=(-g)g$, contrary
to $g\not\in E$. Thus $\frac{1+|E|}{n}\ge 1$, forcing $|E|=n-1$. 
\medskip

It remains to consider the case where $a=1$ and
$B=(-g)^{n-1} T$ is an atom. Since $g\not\in E$, we obtain $|T|\ge 2$ and hence  $\mathsf k(B)>1$. Its support is an LCN-set, and
$d=n-2\le\min\Delta(\supp(B))\le|E|-1$ by
Lemma \ref{l4}.2, forcing $|E|=n-1$.
Set $L=\langle E\rangle$, which contains $g$ and has exponent $n$.
It follows from Lemma \ref{l2}.2   that
\[
      d=n-2\le \min\Delta(\supp(B))\le \mathsf m(L)
      \le\max\{\mathsf r(L)-1,\lfloor n/2\rfloor-1\},
\]
forcing $\mathsf r(L)\ge n-1$. Since $L=\langle E\rangle$,
we also have $\mathsf r(L)\le|E|\le n-1$. Thus $\mathsf r(L)=n-1$.
 By Lemma \ref{l4}.1 and $-g\not\in E$, we have
$\mathsf k(V)\ge1$, while $\mathsf k(B)>1$.
Write $L\cong C_{b_1}\oplus\cdots\oplus C_{b_{n-1}}$ with
$1<b_1\mid\cdots\mid b_{n-1}=n$. Then $\mathsf r(L)-1=n-2\ge \lfloor n/2\rfloor+1$ and we can apply 
Lemma \ref{l5}.1 to $L$ with $d=\mathsf r(L)-1$.
Therefore
\[
 \mathsf k(V)=1+s d/b_1\quad \text{ and }\quad  \mathsf k(B)=1+t d/b_1,
 \quad  \text{ for some  } s\in \mathbb N_0\text{ and } t\in\mathbb N.
\]
Since $b_1\mid n$ and 
\[
\frac{d}{n}=\left(\frac{n-1}{n}+\mathsf k(T)\right)-\left(\frac{1}{n}+\mathsf k(T)\right)=\mathsf k(B)-\mathsf k(V)=\frac{(t-s)d}{b_1},
\]
 we obtain $b_1=n$ and hence 
$L=\langle E\rangle \cong C_{n}^{n-1}$. It follows from $|E|=n-1$ that the elements of $E$ form a basis of $L$ and hence are independent. It suffices to show $\sum_{e\in E}e\in \{g,-g\}$.

Suppose $T=\prod_{e\in E}e^{k_e}$, where $k_e\in [0, n-1]$. 
If $\mathsf k(V)=1$, then $1+\sum_{e\in E} k_e=n$ and $|E|=n-1$ force every $k_e=1$,
implying $\sum_{e\in E}e=-g$.
Suppose $\mathsf k(V)>1$. Consider either $F=E\cup\{g\}$ or
$F=E\cup\{-g\}$. Since $\mathsf k(V)>1$ and $\mathsf k(B)>1$, it follows from Lemma \ref{l4}.2 that $F$ is an LCN-set.
Let $A\in\mathcal A(F)$ with $|\supp(A)|\ge 2$. Since $E$ is independent, the minimality of $E$ implies that $\supp(A)=F$. If $\mathsf k(A)=1$, then $|A|\ge |F|=n$ implies that $A=\prod_{f\in F}f$, whence $A$ divides either $V$ or $B$, a contradiction to the fact that both $V$ and $B$ are atoms of cross number larger than one.  Therefore 
\begin{equation}\label{l5.3-e-1}
\supp(A)=F \quad \text{ and }\quad\mathsf k(A)>1 \quad \text{ for every } A\in \mathcal A(F) \text{ with }|\supp(A)|\ge 2.
\end{equation}

Let
\[
c(T)=\{e\in E\colon 2\mathsf v_e(T)\ge n\}.
\]
Then there exists  $T_2\in \mathcal F(E)$  such that  $$T^2=\left(\prod_{e\in c(T)}e^{\ord(e)}\right)T_2, \quad \text{ with }\mathsf v_e(T_2)<n \text{ for every }e\in E\,.$$
Then both $g^2T_2$ and
$(-g)^{n-2}T_2$ are zero-sum. 
Let $W=g^2T_2$ with $F=E\cup \{g\}$ or $W=(-g)^{n-2}T_2$ with $F=E\cup \{-g\}$, and let $\overline{W}=W^{-1}\prod_{f\in F}f^n$.
Choose $s\in \mathsf L(W)$ and $t\in \mathsf L(\overline{W})$. It follows from \eqref{l5.3-e-1} that
both $W$ and $\overline{W}$ are products of atoms with cross number at least one, and one of the atoms in the products has cross number greater than one. 
Thus $2\le s+t<\mathsf k(W\cdot \overline{W})=\mathsf k(\prod_{f\in F}f^n)=|F|=n$.
Since $\{n, s+t\}\subset \mathsf L(\prod_{f\in F}f^n)$,  $d=n-2$ divides $n-(s+t)$, whence $s=t=1$.
Therefore $W$ is an atom, i.e., 
\[
\text{both\  $g^2T_2$\  and\ 
	$(-g)^{n-2}T_2$\  are atoms. }
\]
Note that $V^2=(g^2T_2)\cdot \left(\prod_{e\in c(T)}e^{\ord(e)}\right)$ and $B^2=(-g)^n\cdot ((-g)^{n-2}T_2)\cdot \left(\prod_{e\in c(T)}e^{\ord(e)}\right)$. It follows that 
\[
2\equiv 1+|c(T)|\equiv 2+|c(T)| \pmod d\,,
\]
whence $d=n-2$ divides $1$, a contradiction.
\end{proof}

\begin{proposition}\label{p5}
Let $n\ge6$ be an even integer with at least two distinct prime divisors.
If $G=C_2\oplus C_{n_2}\oplus\cdots\oplus C_{n_n}$ with
$2\mid n_2\mid\cdots\mid n_n=n$, then
\[
 \rho^*(G,n-2)=\mathsf K(G,n-2).
\]
\end{proposition}
\begin{proof}
Set $d=n-2$. Then $d\in \Delta_1(G)$. Lemma \ref{l6}.1 gives
$$\rho^*(G,d)\ge\mathsf K(G,d)\ge 1+d/2=n/2\,.$$
Let $G_0$ be a subset with $\Delta(G_0)\ne\emptyset$
and $d\mid\min\Delta(G_0)$. Then $\min\Delta(G_0)=d$, since
$\max\Delta^*(G)=n-1<2d$ by Lemma \ref{l2}.2.
It suffices to show $\rho(G_0)\le \mathsf K(G,n-2)$.

If $G_0$ is an LCN-set, Lemma \ref{l6}.2
gives $\rho(G_0)\le\mathsf K(G,n-2)$.
Suppose that $G_0$ is not an LCN-set. In view of Lemma \ref{l6}.4, we may assume  $0\not\in G_0$.
 By Lemma \ref{l4}.1, we may
choose an atom $g(-g)$ of cross number less than one, where $g\in G_0$ with $\ord(g)=n$.
We prove the bound by induction on $|G_0|$. By Lemma \ref{l10}.2, we may write
$G=\langle g\rangle\oplus K$.
Let
\[
\text{$R=G_0\setminus\{g,-g\}$ \quad \text{ and }\quad $J=\langle g\rangle\cap\langle R\rangle$.}
\] 

If $J=0$, then
$\mathcal B(G_0)=\mathcal B(\{g,-g\})\times\mathcal B(R)$.
We have $\rho(\{g,-g\})=n/2\le \mathsf K(G, d)$.
If $\Delta(R)=\emptyset$, then the required bound follows. If $\Delta(R)\neq\emptyset$, then
$d\mid\min\Delta(R)$ and in view of $|R|<|G_0|$,  the
induction hypothesis bounds $\rho(R)\le \mathsf K(G, d)$.
Thus $\rho(G_0)\le\mathsf K(G,d)$.

 Suppose $J\ne0$. 
 Set 
 \begin{align*}
 	&\text{$J_p, K_p, \langle g\rangle_p, G_p$ be the $p$-primary components of $J, K, \langle g\rangle, G$ respectively,}\\
 	\text{ and\quad }&\text{$\pi_p\colon G\rightarrow G_p$ be the projection for every prime $p\mid n$.}
 \end{align*}
For odd primes $p$, we have $\mathsf r(K_p)\le n-2$, and also
$\mathsf r(2\cdot K_2)\le n-2$ if $\mathsf v_2(n)>1$ (note that  $\mathsf v_2(n)=1$ implies $2\cdot K_2=0$). 

Suppose $J_p=0$ for all odd $p$ and $2\cdot J_2=0$. Then $2\cdot J=0$.
As $0\ne J\subset \langle g\rangle$, we obtain $|J|=2$, whence $J=\{0, (n/2)g\}$. Choose an inclusion-minimal
$E\subset R$ with $\langle g\rangle\cap\langle E\rangle\ne0$,
and choose a zero-sum free sequence $T\in \mathcal F(E)$ such that 
 $g^{n/2} T$ is zero-sum. Then both $g^{n/2} T$ and $(-g)^{n/2} T$ are  atoms with $\supp(T)=E$,  since $E$ is inclusion-minimal,
$J=\{0,(n/2)g\}$, and $T$ is zero-sum free.
The equation
\[
 (g^{n/2} T)\cdot (-g)^n=(g\cdot(-g))^{n/2}\cdot ((-g)^{n/2} T)
\]
implies that  $d=n-2$ divides $n/2-1=(n-2)/2<n-2$, a contradiction.

Thus either $J_p\ne0$ for an odd prime $p$ or $2\cdot J_2\ne0$.
Note that $G_p=\langle g\rangle_p\oplus K_p$, $2\cdot G_2= 2\cdot \langle g\rangle_2\oplus 2\cdot K_2$, $J_p=\langle g\rangle_p\cap \langle \pi_p(R)\rangle$, and $2\cdot J_2\subset 2\cdot \langle g\rangle_2\cap 2\cdot \langle \pi_2(R)\rangle$.
Applying Lemma \ref{l10}.1 to either $G_p=\langle g\rangle_p\oplus K_p$ with $\langle g\rangle_p\cap \langle \pi_p(R)\rangle\neq 0$ or $2\cdot G_2=2\cdot \langle g\rangle_2\oplus 2\cdot K_2$ with $2\cdot \langle g\rangle_2\cap \langle 2\cdot \pi_2(R)\rangle\neq 0$, we obtain that 
there exists a subset $E\subset R$ with $|E|\le \max\{\mathsf r(K_p), \mathsf r(2\cdot K_2)\}+1\le n-1$ such that 
$\langle g\rangle\cap\langle E\rangle\ne0$.
By Lemma \ref{l13}, we obtain that $E=\{e_1,\ldots,e_{n-1}\}$ is independent with $\langle E\rangle \cong C_n^{n-1}$ and $g=\sum_{i=1}^{n-1}e_i$ 
after changing its sign, whence 
\[
 G\cong C_2\oplus C_n^{n-1}\quad \text{ and }\quad G_0\supset \{e_1,\ldots,e_{n-1},g,-g\},\qquad \text{ with }g=\sum_{i=1}^{n-1} e_i.
\]
We now determine which additional elements are possible.
If $$h=\sum_{j=1}^{n-1} c_j e_j\in(G_0\cap \langle E\rangle)\setminus \{e_1,\ldots,e_{n-1},g,-g\} \text{ with } c_i\in [1, n-1] \text{ for some }i\in [1,n-1]\,,$$ 
then $c_i g\in\langle h,e_j:j\ne i\rangle$.
Lemma \ref{l13} forces
$h=e_i$ or $h=e_i-2g$. The first is excluded by the choice of $h$.
In the second case, $h=e_i-2\sum_{j=1}^{n-1} e_j$ has a nonzero coefficient
of every $e_j$ with $j\ne i$. Applying the same argument to $h$ with such
a $j$ would give $h=e_j$ or $h=e_j-2g$, both contrary to the
independence of  $E$. Thus $$G_0\cap \langle E\rangle=\{e_1,\ldots,e_{n-1},g,-g\}\,.$$

If $h\in G_0\setminus \langle E\rangle$ with $2h\ne0$, then write
$2h=\sum_{i=1}^{n-1} c_i e_i$ and choose $i$ with $c_i\in [1, n-1]$. Again
$c_i g\in\langle h,e_j:j\ne i\rangle$ and Lemma \ref{l13} forces 
$\pm g=h+\sum_{j\ne i}e_j$, implying  $h\in \langle E\rangle$, a contradiction.
Thus every element of $G_0\setminus \langle E\rangle$ has order two.
Note that $ G\cong C_2\oplus \langle E\rangle$.
If $h_1,h_2\in G_0\setminus \langle E\rangle$ are distinct, then
$h_1+h_2\in \langle E\rangle$ has order $2$ and hence $h_1+h_2=\sum_{i\in I}(n/2)e_i$ for a nonempty
$I\subset[1,n-1]$. If $I$ contains distinct $i,j$, then
$(n/2)g$ belongs to the subgroup generated by $h_1,h_2$ and
the $n-3$ elements $e_k$ with $k\notin\{i,j\}$.
Lemma \ref{l13} would force $h_1$ and $h_2$ to have
order $n$, a contradiction. If $I=\{i\}$, then
$h_1 h_2 e_i^{n/2}$ is an atom with $\left(h_1 h_2 e_i^{n/2}\right)^2=h_1^2\cdot h_2^2\cdot e_i^n$, whence 
$d$ divides $1$, a contradiction. Therefore 
\[
|G_0\setminus\langle E\rangle|\le 1\,.
\]
Any element of $G_0\setminus \langle E\rangle$ is independent of $\langle E\rangle$ since it has order $2$.
Thus $\mathcal B(G_0)=\mathcal B(\{e_1,\ldots,e_{n-1},g,-g\})$ or
$\mathcal B(G_0)=\mathcal B(\{e_1,\ldots,e_{n-1},g,-g\})\times\mathcal B(\{h\})$ with
$\ord(h)=2$. In either case, $$\rho(G_0)=\rho(\{e_1,\ldots,e_{n-1},g,-g\})\,.$$
Since $\min\Delta(\{e_1,\ldots,e_{n-1},g,-g\})=n-2=\mathsf r(\langle E\rangle)-1$ and $n$ is not a
prime power, Proposition \ref{p4}.2 gives
\[
\rho(G_0)
\le\rho^*(\langle E\rangle,n-2)=\mathsf K(\langle E\rangle,n-2)\le\mathsf K(G,n-2).
\]
\end{proof}

\section{Proof of the characterization theorem}\label{s6}
\begin{proof}[Proof of Theorem \ref{t2}]
Let $H\cong C_n^r$ and suppose $\mathcal L(G)=\mathcal L(H)$.
By Lemma \ref{l3}, we have $\mathsf D(H)=\mathsf D(G)\ge4$.
If $r\le n-3$, the assertion follows from Lemma \ref{l8}.1.
If $r\ge n-1$, then $\mathsf D(H)\ge 4$ implies that $r\ge2$, and the assertion now follows from Theorem \ref{t1}  and Lemma \ref{l8}.3.

It remains to consider $r=n-2$. Since $\mathsf D(H)\ge4$,
we have $n\ge4$ and $r\ge2$. Write
\[
 G\cong C_{m_1}\oplus\cdots\oplus C_{m_R},\qquad \text{ with }\quad 
 1<m_1\mid\cdots\mid m_R,\qquad \text{ and }\quad  m=m_R.
\]
By Lemma \ref{l3},
\begin{equation}\label{e14}
 \max\{m-2,R-1\}=n-2\,,
\end{equation}
whence $m\le n$ and $R\le n-1$. If $R\le n-2$, then
\eqref{e14} implies $m=n$. Since every $m_i$ divides $n$,
the group $G$ is isomorphic to a subgroup of $C_n^{n-2}=H$. The assertion now follows from 
Lemma \ref{l3}.
Thus we may assume  $$R=n-1\,.$$
By Lemma \ref{l2}.3, we have $n-2=R-1\in\Delta^*(G)$.
Lemma \ref{l3} gives $n-2\in\Delta_1(G)=\Delta_1(H)$.
Since
\[
 n-2=\mathsf r(H)\ge\max\{\mathsf r(H),\lfloor n/2\rfloor,
                         \lfloor(\mathsf r(H)+1)/2\rfloor\},
\]
Lemma \ref{l7}.1 yields $\rho^*(H,n-2)=n/2$.
By Lemmas \ref{l3} and \ref{l6}.1,
\[
 \frac n2=\rho^*(H,n-2)=\rho^*(G,n-2)
 \ge\mathsf K(G,n-2)\ge1+(n-2)(1-1/m_1)\ge\frac n2.
\]
Consequently,
\begin{equation}\label{e15}
 m_1=2.
\end{equation}

Suppose $n\in \{4,5\}$. Since $m_1=2$ and $m\le n$,
we have $m\in\{2,4\}$. If $n=4$, then $R=3$ and hence $G$ is isomorphic to a subgroup  of $C_2\oplus C_4\oplus C_4$. By Lemma \ref{l1}.2,  we infer that $\mathsf D(G)\in \{4,6,8\}$ and $\mathsf D(H)=\mathsf D(C_4^2)=7$, a contradiction to $\mathsf D(G)=\mathsf D(H)$.
If $n=5$, then $R=4$ and hence $G$ is isomorphic to a subgroup of $C_2\oplus C_4^3$. By Lemma \ref{l1}.2,  we infer that $\mathsf D(G)\le \mathsf D(C_2\oplus C_4^3)=11$ and $\mathsf D(H)=\mathsf D(C_5^3)=13$, a contradiction to $\mathsf D(G)=\mathsf D(H)$.
Thus we may assume that $$n\ge6\,.$$

Since $r=n-2\ge 4$, Lemma \ref{l2}.3 gives
$n-3=r-1\in\Delta_1(H)$. It follows from 
Theorem \ref{t1} and Proposition \ref{p3} that
\begin{equation}\label{e16}
 \rho^*(H,n-3)=\mathsf K(H,n-3)=1+\frac{a(n-3)}{n},\qquad \text{ where }
 a=n\sum_{p\mid n}(1-p^{-\mathsf v_p(n)})\,.
\end{equation}
 If $m_2=n$, then $H$ is isomorphic to a subgroup of $G$ and hence the assertion follows from Lemma \ref{l3}.
Thus we may assume that 
\begin{equation}\label{m_2}
m_2<n.
\end{equation}
If $m=n-1$ is a power of two, then $R=n-1=m$ and Lemma \ref{l1} implies 
\[
 \mathsf D(G)=1+\sum_{i=1}^{m}(m_i-1)
 \le2+(m-1)^2<m^2-m+1= 1+(n-2)(n-1)\le \mathsf D(H),
\]
 a contradiction to $\mathsf D(G)=\mathsf D(H)$.  Thus 
 \[
 \text{ $m$ is not a power of $2$ when $m=n-1$.} 
 \]

We now show that $\rho^*(G,n-3)=\mathsf K(G,n-3)$ in every remaining case.
If $m\le n-2$, then $n-3>m-2$, so Lemma \ref{l6}.3 gives
this equality. If $m=n$, this equality follows from Proposition \ref{p3},
since $R=n-1$ and $m_1=2$. Finally, if $m=n-1$ is not a power of two, then $2\mid m$ and   Proposition \ref{p5} implies that $\rho^*(G,n-3)=\rho^*(G, m-2)=\mathsf K(G, m-2)=\mathsf K(G,n-3)$. Since $\lfloor m/2\rfloor \le n-3\le R-1$, it follows from  Lemma \ref{l5}.1 that
\[
\rho^*(G,n-3)=\mathsf K(G,n-3)=1+\frac{b(n-3)}{m_2},
\qquad\text{for some }b\in\mathbb N_0.
\]
By Lemma \ref{l3} and \eqref{e16}, we obtain
\[
 1+\frac{a(n-3)}{n}=\rho^*(H,n-3)=\rho^*(G,n-3)=1+\frac{b(n-3)}{m_2}\,,
\]
whence $a/n=b/m_2$. By \eqref{e16}, it is easy to conclude that $\gcd(a,n)=1$, whence $n\mid m_2$, contrary to \eqref{m_2}.
This completes the proof.
\end{proof}

\section{Concluding remarks and further problems}\label{s7}
Theorem \ref{t1} determines $\mathsf K(G,r-1)$ when
$r\ge\max\{2,\exp(G)-2\}$. In this range, its value depends only
on the rank and the order of the smallest cyclic factor in the
decomposition $G\cong C_{n_1}\oplus\cdots\oplus C_{n_r}$ with
$1<n_1\mid\cdots\mid n_r$. Theorem \ref{t2} resolves the
Characterization Problem for homocyclic finite abelian groups with
Davenport constant at least four. We conclude with three related
questions.

\medskip\noindent\textbf{1. Characterization beyond homocyclic groups.}
The general Characterization Problem asks whether
\[
 \mathcal L(G)=\mathcal L(G')\quad\Longrightarrow\quad G\cong G'
\]
for finite abelian groups $G,G'$ with $\mathsf D(G)\ge4$
\cite{Sc09b,Zhong,GGZ2026}. The groups
$C_m\oplus C_n^r$, where $1<m<n$ and $m\mid n$, provide a
natural family in which to extend the present results.
For which of these groups can comparisons of refined elasticities
determine $m$, $n$, and $r$? More generally, which integers in the
decomposition of a finite abelian group into cyclic factors can be
recovered from its refined elasticities?

\medskip\noindent\textbf{2. Cross numbers at smaller ranks and distances.}
Let $G\cong C_{n_1}\oplus\cdots\oplus C_{n_r}$ with
$1<n_1\mid\cdots\mid n_r$ and $2\le r<\exp(G)-2$.
Determine $\mathsf K(G,r-1)$ and identify the groups for which the
formula in Theorem \ref{t1} remains valid. One may also ask for
$\mathsf K(G,d)$ when $d\in\Delta_1(G)$ and $d<r-1$.
Under the hypotheses of Lemma \ref{l5}.1, the denominator of
$\mathsf k(U)-1$ divides $n_{r-d}$ for every atom
$U$ of cross number greater than one over a subset whose minimum
distance is divisible by $d$. Can this restriction be combined
with estimates for the primary components to determine further
values of $\mathsf K(G,d)$? The argument in Section \ref{s3} also uses
Lemma \ref{l9}, whose hypothesis is
$\min\Delta(G_0)\ge\exp(G)-3$. Bounds for the maximum of the
factorization lengths under weaker hypotheses may allow the same
approach at smaller distances.

\medskip\noindent\textbf{3. Equality of refined elasticity and cross number.}
For every $d\in\Delta_1(G)$, we have
$\rho^*(G,d)\ge\mathsf K(G,d)$. Propositions \ref{p3}, \ref{p4},
and \ref{p5} establish equality in several cases, while
Proposition \ref{p4} also records a family with strict inequality.
Determine further pairs $(G,d)$ for which equality holds.
The structural arguments in Sections \ref{s4} and \ref{s5}
concern atoms of cross number less than one at minimum distances
$\exp(G)-3$ and $\exp(G)-2$, respectively. Their supports have
three and two elements. More generally, Lemma \ref{l4}.1
determines their lengths when the minimum distance is at least
$\lfloor\exp(G)/2\rfloor$. Descriptions of the subgroups generated
by larger supports, and of their intersections with the subgroups
generated by the remaining elements, may yield further elasticity
comparisons and applications to the Characterization Problem.

\bigskip
\noindent{\bf AI Declaration}
No artificial intelligence (AI) tools, large language models or generative AI systems were used in the research, analysis, preparation or writing of this manuscript.

\end{document}